\documentclass{article}

\usepackage{amsthm,amsmath,amssymb}
\usepackage{enumitem}
\usepackage{geometry}
\usepackage{caption}
\usepackage{graphicx}
\usepackage{fontenc}
\usepackage[utf8]{inputenc}
\usepackage{array}
\usepackage[greek, english]{babel}
\usepackage{booktabs}
\usepackage{cite} 
\usepackage{float} 
\usepackage[bottom]{footmisc} 
\usepackage{kvoptions}
\usepackage{multicol} 
\usepackage{nag} 
\usepackage{newtxmath} 
\usepackage{newtxtext} 
\usepackage{pdf14} 
\usepackage{pdftexcmds}  
\usepackage{ragged2e} 
\usepackage{url} 
\usepackage{xcolor} 
\usepackage{xpatch} 
\usepackage{zref-base}
\usepackage{hyperref}

\numberwithin{equation}{section}

\newcommand{\R}{\ensuremath{\mathbb R}}

\newcommand{\from}{\colon}
\newcommand{\id}{\ensuremath{\mathrm{id}}}

\newcommand{\abs}[1]{\ensuremath{\left\lvert #1\right\rvert}}

\newcommand{\cbullet}{\,\raisebox{1pt}{\(\scriptscriptstyle\bullet\)}\,}

\renewcommand{\implies}{\ensuremath{\Rightarrow}}
\renewcommand{\iff}{\ensuremath{\Leftrightarrow}}

\newtheorem{lemma}{Lemma}[section]
\newtheorem{proposition}{Proposition}[section]
\newtheorem{theorem}{Theorem}[section]
\newtheorem{corollary}{Corollary}[section]
\newtheorem{remark}{Remark}[section]
\newtheorem*{claim*}{Claim}

\newtheorem{definition}{Definition}[section]

\newtheorem{theoremA}{Theorem}[section]

\begin{document}

\title{Semialgebraic Lipschitz equivalence of polynomial functions}

\author{Sergio Alvarez}

\date{}

\maketitle

\begin{abstract}
\noindent
We investigate the classification of quasihomogeneous polynomials in two variables with real coefficients under semialgebraic bi-Lipschitz equivalence  in a neighborhood of the origin in \(\R^2\). Building on the work of Birbrair, Fernandes, and Panazzolo,  our approach is based on reducing the problem to the Lipschitz classification of associated single-variable polynomial functions, called height functions. We establish conditions under which semialgebraic bi-Lipschitz equivalence of quasihomogeneous polynomials corresponds to the Lipschitz equivalence of their height functions. To achieve this, we develop the theory of \(\beta\)-transforms and inverse \(\beta\)-transforms. As an application, we examine a family of quasihomogeneous polynomials previously used by Henry and Parusiński to show that the bi-Lipschitz equivalence of analytic function germs \((\R^2,0)\to(\R,0)\) admits continuous moduli. Our results show that semialgebraic bi-Lipschitz equivalence of real quasihomogeneous polynomials in two variables also admits continuous moduli.

\end{abstract}


\section{Introduction}
\label{section: intro}

In \cite{HP1}, Henry and Parusi\'{n}ski showed that the bi-Lipschitz classification of complex analytic function germs admits continuous moduli. This fact had not been observed before and, interestingly, it contrasts with the bi-Lipschitz equivalence of complex analytic set germs, which does not admit moduli (see \cite{M}). The moduli space of bi-Lipschitz equivalence of function germs is not yet fully understood. However, recent progress by Câmara and Ruas has advanced the understanding of the moduli space of bi-Lipschitz equivalence of quasihomogeneous function germs, in the complex case (see \cite{CR}).

In \cite{HP2}, Henry and Parusi\'{n}ski showed that the bi-Lipschitz classification of real analytic function germs admits continuous moduli.  Then, in \cite{BFP}, Birbrair, Fernandes and Panazzolo described the semialgebraic bi-Lipschitz moduli in what they termed ``the simplest possible case'': quasihomogeneous polynomial functions defined on the Hölder triangle 
\(T_\beta \coloneqq \{(x,y)\in\R^2: x\geq 0, \,0\leq y\leq x^\beta\}\). 
Independently, the particular case of weighted homogeneous polynomial functions of two real variables has been considered by Koike and Parusi\'{n}ski in \cite{KP}.

In this paper, we address the problem of classifying 
quasihomogeneous polynomials in two variables with real coefficients modulo semialgebraic bi-Lipschitz equivalence. Our main goal is to extend the results obtained in \cite{BFP} for the classification of germs of functions defined on the Hölder triangle \(T_\beta\) to germs of functions defined on the whole plane.

For simplicity, we restrict our discussion (without loss of generality) to quasihomogeneous polynomials \(F(X,Y)\in\R[X,Y]\) in which the weight of \(Y\) is greater than the weight of \(X\). Such polynomials are precisely those satisfying a relation of the form \(F(tX,t^{\beta}Y) = t^dF(X,Y)\), where \(\beta\) is a rational number greater than \(1\) and \(d\) is a positive rational number. Thus, following \cite{BFP}, we refer to such polynomials as {\it \(\beta\)-quasihomogeneous of degree \(d\)}. The rational number \(d\) is called the {\it \(\beta\)-quasihomogeneous degree of \(F\)}. In this paper, we consider only the case where \(d\) is a positive integer to keep the discussion focused.

Given two \(\beta\)-quasihomogeneous polynomials \(F(X,Y)\) and \(G(X,Y)\) of the same \(\beta\)-quasihomogeneous degree \(d\), we aim to determine whether there exists a germ of semialgebraic bi-Lipschitz homeomorphism \(\Phi\from (\R^2,0)\to(\R^2,0)\) such that \(G\circ\Phi = F\). In the affirmative case, we say that \(F\) and \(G\) are\/ {\it \({\cal R}\)-semialgebraically Lipschitz equivalent}. Following the approach taken in \cite{BFP}, we seek to reduce this problem to determining whether two given polynomial functions of a single variable are Lipschitz equivalent --- we say that two polynomial functions \(f,g\from\R\to\R\) are {\it Lipschitz equivalent}, denoted \(f \cong g\), if there exist a bi-Lipschitz homeomorphism 
\(\phi\from\R\to\R\) and a constant \(c > 0\) such that \(g\circ\phi = c f\).

The Lipschitz classification problem for polynomial functions of a single variable is solved in Section \ref{section: Lipschitz equivalence on a single variable} (see Theorem \ref{thm: no critical points}, Theorem \ref{thm: only one critical point}, and Theorem \ref{thm: at least 2 critical points}). Again, we follow the approach taken in \cite{BFP}, which consists of comparing both the values and the multiplicities of the given polynomial functions at their critical points.
The reduction to the single variable case is carried out in Section \ref{section: R-semialg Lip equiv of beta-qh polynomials}, under fairly general conditions. Still following the methodology of \cite{BFP}, we associate with each 
\(\beta\)-quasihomogeneous polynomial \(F(X,Y)\in\R[X,Y]\) a pair of 
polynomial functions \(f_+,f_-\from\allowbreak{\R\to\R}\), called the {\it height functions} of \(F\). Then, we consider the following questions:
\begin{enumerate}[label=(\Roman*)]
\item
\label{question 1}
Suppose that two given \(\beta\)-quasihomogeneous polynomials 
\(F,G\in\R[X,Y]\) of degree \(d\) are \({\cal R}\)-semialgebraically Lipschitz equivalent. Is it possible to arrange their height functions into pairs of Lipschitz equivalent functions
(i.e. is it true that either \(f_+\cong g_+\) and \(f_-\cong g_-\), or
\(f_+\cong g_-\) and \(f_-\cong g_+\))?
\item 
\label{question 2}
Suppose that the height functions of two given \(\beta\)-quasihomogeneous polynomials \(F,G\in\R[X,Y]\) of degree \(d\) can be arranged into pairs of Lipschitz equivalent functions. Is it true that \(F\) and \(G\)\, are \({\cal R}\)-semialgebraically Lipschitz equivalent?
\end{enumerate}

We establish conditions under which the answer to the first question is affirmative (see Theorem \ref{thm: equivalent polynomials, Lipschitz equivalent height functions}). Additionally, we provide some conditions under which the answer to the second question is affirmative (see Theorem \ref{thm: Sufficient conditions for R-semialg. Lip. equivalence}, Corollary \ref{cor: Sufficient conditions for R-semialg. Lip. equivalence}, and Corollary \ref{cor: r odd, s even}). These are our main results on \({\cal R}\)-semialgebraic Lipschitz equivalence of \(\beta\)-quasihomogeneous polynomials. Together with our results on Lipschitz equivalence of polynomial functions of a single variable (namely, Theorems \ref{thm: no critical points}, \ref{thm: only one critical point}, and \ref{thm: at least 2 critical points}), they allow us to determine, under fairly general conditions, whether two given \(\beta\)-quasihomogeneous polynomials are \({\cal R}\)-semialgebraically Lipschitz equivalent.

In \cite{BFP}, the questions stated above were both answered affirmatively in the case where the equivalence is restricted to the Hölder triangle \(T_\beta\), assuming that the given \(\beta\)-quasihomogeneous polynomials vanish identically on \(\partial T_\beta\) and do not vanish at the interior points of 
\(T_\beta\). Here, building on the methods of Birbrair, Fernandes, and Panazzolo, we develop a broader framework that extends their results on \(\beta\)-quasihomogeneous polynomials to a neighborhood of the origin in the whole plane. This is the theory of \(\beta\)-transforms and inverse \(\beta\)-transforms, introduced in Sections \ref{subsection: beta-isomorphisms and the beta-transform},
\ref{subsection: proto-transitions}, \ref{subsection: beta-transitions and the inverse beta-transform}, and \ref{subsection: shifting from proto to beta}. Our main results on \({\cal R}\)-semialgebraic Lipschitz equivalence of \(\beta\)-quasihomogeneous polynomials in two variables are proved using this theory.

In Section \ref{subsection: Henry and Parusinski's example}, we apply our main results on \({\cal R}\)-semialgebraic Lipschitz equivalence of 
\(\beta\)-quasihomogeneous polynomials in two variables to examine a specific family of quasihomogeneous polynomials, previously used in \cite{HP2} to show that the bi-Lipschitz equivalence of analytic function germs \((\R^2,0)\to(\R,0)\) admits continuous moduli. As an immediate consequence of our results, we establish that the \({\cal R}\)-semialgebraic Lipschitz equivalence of \(\beta\)-quasihomogeneous polynomials in two variables also admits continuous moduli.

\section{Lipschitz equivalence of polynomial functions of a single variable}
\label{section: Lipschitz equivalence on a single variable}

\begin{definition}
Two polynomial functions \(f,g\from\R\to\R\) are {\it Lipschitz equivalent}, denoted \(f \cong g\), if there exist a bi-Lipschitz homeomorphism 
\(\phi\from\R\to\R\) and a constant \(c > 0\) such that \(g\circ\phi = c f\).
\end{definition}

In this section, we show how to determine whether two given polynomial functions \(f,g\from\R\to\R\) are Lipschitz equivalent and, in the case where \(f\) and \(g\) are nonconstant Lipschitz equivalent polynomial functions, we investigate the asymptotic behavior of the function \(\phi\) relating \(f\) and \(g\).

\subsection{Lipschitz classification of polynomial functions of a single variable}
\label{subsection: Lipschitz equivalence of polynomial functions of a single variable}

In this subsection, we provide effective criteria to determine whether any two given polynomial functions \(f,g\from\R\to\R\) are Lipschitz equivalent. The case of constant functions is trivial: a polynomial function \(f\) is Lipschitz equivalent to a constant function \(g\) if and only if \(f\) is constant and has the same sign as \(g\) (positive, negative, or zero). Thus, we focus on nonconstant polynomial functions.  In this case, if \(f\) and \(g\) are Lipschitz equivalent, then the function \(\phi\) that relates \(f\) and \(g\) is necessarily semialgebraic. This follows from the next lemma, which is an adaptation of \cite[Lemma~3.1]{BFP}.

\begin{lemma}
\label{lemma: semialgebraicity of phi}
Let \(f,g\from\R\to\R\) be nonconstant polynomial functions. 
If \(\phi\from\R\to\R\) is a homeomorphism such that \(g\circ\phi = f\), then 
\(\phi\) is semialgebraic.
\end{lemma}

\begin{proof}
Let \(s_1<\cdots< s_p\) be the critical points of \(g\). Then we have 
\begin{equation*}
\R = (-\infty,s_1]\cup [s_1,s_2]\cup\cdots\cup [s_{p-1},s_p]\cup[s_p,+\infty),
\end{equation*}
and \(g\) is monotone and injective on each of the intervals 
\((-\infty,s_1], [s_1,s_2],\ldots, [s_{p-1},s_p]\), and \([s_p,+\infty)\).
Let \(t_i \coloneqq \phi^{-1}(s_i)\), for \(i = 1,\ldots, p\). 

Suppose that \(\phi\) is an increasing homeomorphism; the case where \(\phi\) is decreasing is analogous. Then we have \(t_1 < \cdots < t_p\). On each of the intervals 
\((-\infty,t_1],[t_1,t_2],\ldots,[t_{p-1},t_p]\), and \([t_p,+\infty)\), we have
\(\phi = g^{-1}\circ f\). Since \(f\) and \(g\) are polynomial functions, it follows that each of the restrictions 
\(\phi |_{(-\infty,t_1]},\allowbreak \phi |_{[t_1,t_2]},\allowbreak \ldots,\allowbreak \phi |_{[t_{p-1},t_p]}\), and \(\phi |_{[t_p,+\infty)}\) is a semialgebraic function. Thus, \(\phi\) is a semialgebraic function.
\end{proof}

\begin{lemma}
\label{lemma: bi-Lipschitz invariance of degree}
Let \(f,g\from\R\to\R\) be nonconstant polynomial functions. If \(f\) and \(g\) are Lipschitz equivalent, then \(\deg f = \deg g\).
\end{lemma}

\begin{proof}
Let \(f(t) = \sum_{i=0}^d a_it^i\) and \(g(t) = \sum_{i=0}^e b_it^i\), where 
\(a_d,b_e\neq 0\). Suppose that \(f\) and \(g\) are Lipschitz equivalent, so that \(g\circ \phi = cf\), for some bi-Lipschitz function \(\phi\from\R\to\R\) and some constant \(c > 0\). We must show that \(d = e\).

Let \(\lambda\coloneqq \lim_{t\to+\infty}\phi(t)/t\).
This limit is well-defined in the extended real line, because \(\phi\) is semialgebraic (see Lemma \ref{lemma: semialgebraicity of phi}). 
Also, since \(\phi\) is bi-Lipschitz, \(\lambda\) is a nonzero real number.
Since \(\lim_{\abs{t}\to+\infty}\abs{\phi(t)} = +\infty\), 
we have:
\begin{equation*}
\lim_{\abs{t}\to+\infty}\frac{g(\phi(t))}{\phi(t)^e} 
= \lim_{\abs{t}\to+\infty}\frac{g(t)}{t^e}
= b_e
\end{equation*}
Then, since
\(c f = g\circ\phi\), we have:
\begin{equation*}
\lim_{\abs{t}\to+\infty}\frac{f(t)}{t^e} = \frac{1}{c}\cdot\lim_{\abs{t}\to+\infty}\frac{g(\phi(t))}{t^e}
=\frac{1}{c}\cdot\lim_{\abs{t}\to+\infty}\frac{g(\phi(t))}{\phi(t)^e}\cdot
   \lim_{\abs{t}\to+\infty} \left(\frac{\phi(t)}{t}\right)^e
= \frac{b_e\cdot \lambda^e}{c}
\end{equation*}
Since \(b_e\cdot\lambda^e/c\) is a nonzero real number, it follows that \(d = e\).
\end{proof}

For the next result, we need to introduce the concept of multiplicity of a nonconstant analytic function at a point of its domain.

\begin{definition}
Let \(f\from I\to \R\) be a nonconstant analytic function defined on an open interval \(I\subseteq \R\). The multiplicity of \(f\) at a point \(t_0\in I\) is the integer \(r\geq 1\) \textup{(}uniquely determined by \(f\) and \(t_0\)\textup{)} \,for which there exists an analytic function \(g\from (t_0-\delta, t_0 + \delta)\to\R\), defined on an open interval \((t_0-\delta, t_0 + \delta)\subseteq I\), such that
\(g(t_0)\neq 0\) and \(f(t) - f(t_0) = (t - t_0)^r\cdot g(t)\), for all \(t\in (t_0-\delta,t_0+\delta)\).
\end{definition}

\begin{remark}
\label{rk: characterization of critical points in terms of multiplicity}
From the Taylor expansion of \(f\) around \(t_0\), it is clear that the multiplicity of \(f\) at \(t_0\) is the smallest integer \(k\geq 1\) for which \(f^{(k)}(t_0)\neq 0\). Thus, \(t_0\) is a critical point of \(f\) if and only if the multiplicity of \(f\) at \(t_0\) is \(\geq 2\).
\end{remark}

\begin{lemma}
\label{lemma: bi-Lipschitz iff bi-analytic}
 Let \(f,g\from\R\to\R\) be two polynomial functions of the same degree 
 \(d\geq 1\), and suppose that \(\phi\from\R\to\R\) is a bijective function such that \(g\circ\phi = cf\), for some constant \(c > 0\). Then the following conditions are equivalent:
\begin{enumerate}[label=\roman*.]
\item \(\phi\) is bi-Lipschitz;
\item The multiplicity of \(f\) at \(t\) is equal to the multiplicity of \(g\) at \(\phi(t)\), for all \(t\in\R\); 
\item \(\phi\) is bi-analytic.
\end{enumerate}
\end{lemma}
 
\begin{proof}
\begin{description}[leftmargin=0cm,itemsep=0.2cm]
\item[\((i)\implies (ii)\colon\)] Pick any point \(t_0\in\R\). Let \(k\) be the multiplicity of \(f\) at \(t_0\), and let \(l\) be the multiplicity of \(g\) at \(\phi(t_0)\). For any pair of functions \(u,v\from\R\to\R\), we write \(u \sim v\) near \(t_0\) to indicate that there exist constants \(A,B>0\) such that
\(A\abs{v(t)}\leq \abs{u(t)} \leq B\abs{v(t)}\), 
for \(t\) sufficiently close to \(t_0\). Then \(f(t) - f(t_0) \sim (t-t_0)^k\) near \(t_0\) and \(g(s) - g(\phi(t_0)) \sim (s-\phi(t_0))^l\) near \(\phi(t_0)\). Since \(g\circ \phi = c f\), this implies that \((t-t_0)^k\sim (\phi(t)-\phi(t_0))^l \) near \(t_0\). And since we are assuming that \(\phi\) is bi-Lipschitz, it follows that \((t-t_0)^k \sim (t-t_0)^l\) near \(t_0\). Therefore, \(k = l\).

\item[\((ii)\implies (iii)\colon\)] Pick any point \(t_0\in\R\). 
Suppose that \(\hat{f}\coloneqq c f\) has multiplicity \(k\) at \(t_0\).
Then there exist an increasing analytic diffeomorphism \(u\from I \to (-\epsilon, \epsilon)\), 
with \(t_0\in I\), and a constant \(\rho\in\R\setminus\{0\}\), 
such that \(u(t_0) = 0\) and 
\(\hat{f}\circ u^{-1}(t) = a + \rho t^k\), for \(\abs{t} < \epsilon\);
where \(a \coloneqq \hat{f}(t_0) = g\circ\phi(t_0)\).
Since we are assuming that condition \((ii)\) holds, the multiplicity of \(g\) at the point \(\phi(t_0)\) is also \(k\). Then there exist an increasing analytic diffeomorphism \(v\from J \to (-\epsilon^\prime, \epsilon^\prime)\), with \(\phi(t_0)\in J\), and a constant \(\sigma\in\R\setminus\{0\}\), such that \(v(\phi(t_0)) = 0\) and 
\(g\circ v^{-1}(t) = a + \sigma t^k\), for \(\abs{t} < \epsilon^\prime\). 
Shrinking the interval \(I\), if necessary, we can assume that \(\phi(I)\subseteq J\). Hence, we can write \(\hat{f}\circ u^{-1}(t) = g\circ v^{-1}(\bar{\phi}(t))\), where 
\(\bar{\phi} \coloneqq 
v\circ \phi\circ u^{-1}\from (-\epsilon,\epsilon)\to (-\epsilon^\prime,\epsilon^\prime)\);
and then it follows that \(\bar{\phi}(t) = \nu t\), where
\(\nu = \pm\abs{\frac{\rho}{\sigma}}^{\frac{1}{k}}\), with positive sign if \(\phi \) is increasing and negative sign if \(\phi \) is decreasing.
In particular, this shows that \(\bar{\phi}\) is analytic.
Therefore, \(\phi |_{I} = v^{-1}\circ \bar{\phi} \circ u\) is analytic; so \(\phi\) is analytic at \(t_0\). Since the point \(t_0\in\R\) is arbitrary, it follows that \(\phi\) is an analytic function. In order to show that \(\phi^{-1}\) is also analytic, note that \(\phi^{-1}\from\R\to\R\) is a bijective function such that \(f\circ\phi^{-1} = c^{-1}g\), satisfying condition (ii) with \(f\) and \(g\) interchanged: the multiplicity of \(g\) at \(t\) is equal to the multiplicity of \(f\) at \(\phi^{-1}(t)\), for all \(t\in\R\). Then, by what we just proved, it follows that \(\phi^{-1}\) is analytic.

\item[\((iii)\implies (i)\colon\)] Suppose that \(\phi\) is bi-analytic. 
Then, in particular, \(\phi\) is a homeomorphism, so \(\phi\) is monotone and 
\(\lim_{\abs{t}\to+\infty}\abs{\phi(t)}~=~+\infty\). Also, by Lemma \ref{lemma: semialgebraicity of phi}, \(\phi\) is a semialgebraic function.
Let \(f(t) = \sum_{i=0}^{d} a_i t^i\) and \(g(t) = \sum_{i=0}^{d} b_i t^i\), 
with \(a_d,b_d\neq 0\). Since \(g\circ\phi = cf\), we have
\begin{equation}
\label{eq: limits to the power d}
\lim_{t\to-\infty}\left(\frac{\phi(t)}{t}\right)^d = 
\lim_{t\to+\infty}\left(\frac{\phi(t)}{t}\right)^d =
c\cdot\frac{a_d}{b_d}.
\end{equation}

Let 
\(l_+ \coloneqq \lim_{t\to+\infty}\phi(t)/t\) and\/  
\(l_- \coloneqq \lim_{t\to-\infty}\phi(t)/t\).
Both these limits are well-defined in the extended real line because \(\phi\) is semialgebraic. It follows from (\ref{eq: limits to the power d}) that we actually have \(l_+,l_-\in\R\setminus\{0\}\) and \(\abs{l_+} = \abs{l_-}\). (Notice that to obtain this last equality from 
(\ref{eq: limits to the power d}), we use the fact that \(d > 0\).)
By L'Hôpital's rule, 
\(\lim_{t\to+\infty} \phi'(t) = l_+\) and \(\lim_{t\to-\infty} \phi'(t) = l_-\). 
(The existence of these limits in the extended real line is guaranteed by the fact that \(\phi'\) is semialgebraic, so L'Hôpital's rule can be applied.) Thus, \(\lim_{t\to+\infty} \abs{\phi'(t)} = \lim_{t\to-\infty} \abs{\phi'(t)} > 0\); so that \(\abs{\phi'}\) can be continuously extended to a positive function defined on the compact space \(\R\cup\{\infty\}\cong {\mathbb P}^1(\R)\). Hence, there exist constants \(A,B > 0\) such that \(A\leq \abs{\phi'(t)}\leq B\), for all \(t\in\R\). 
Therefore, \(\phi\) is bi-Lipschitz.
\end{description}
\end{proof} 

From the last two lemmas, it follows that if two polynomial functions 
\(f,g\from\R\to\R\) are Lipschitz equivalent, then they have the same degree and the same number of critical points. The first assertion is precisely the content of Lemma \ref{lemma: bi-Lipschitz invariance of degree}. The second assertion is an immediate consequence of Lemma \ref{lemma: bi-Lipschitz iff bi-analytic}: if \(\phi\from\R\to\R\) is a bi-Lipschitz homemorphism such that \(g\circ \phi = cf\), for some constant \(c > 0\), then 
\(\phi\) induces a  1-1 correpondence between the critical points of \(f\) and the critical points of \(g\), because it preserves multiplicity (see Remark \ref{rk: characterization of critical points in terms of multiplicity}).

The next results provide effective criteria to determine whether any two given nonconstant polynomial functions \(f,g\from\R\to\R\) of the same degree are Lipschitz equivalent. In Theorem \ref{thm: no critical points}, we address the case where \(f\) and \(g\) have no critical points; in Theorem \ref{thm: only one critical point}, the case where both \(f\) and \(g\) have only one critical point; and in Theorem \ref{thm: at least 2 critical points}, the case where \(f\) and \(g\) have the same number \(p\geq 2\) of critical points.

\begin{theoremA}
\label{thm: no critical points}
Let \(f,g\from\R\to\R\) be polynomial functions of the same degree 
\(d\geq 1\). If \(f\) and \(g\) have no critical points, then \(f\) and \(g\) are Lipschitz equivalent. 
\end{theoremA}

\begin{proof}
If \(f\) and \(g\) have no critical points, then they are both bi-analytic diffeomorphisms. Hence,  \(f = g\circ \phi\), where 
\(\phi \coloneqq g^{-1}\circ f\) is a bi-analytic diffeomorphism. By Lemma \ref{lemma: bi-Lipschitz iff bi-analytic}, \(\phi\) is bi-Lipschitz.
\end{proof}
 
\begin{theoremA}
\label{thm: only one critical point}
Let \(f,g\from\R\to\R\) be polynomial functions of the same degree 
\(d\geq 1\). Suppose that \(f\) has only one critical point \(t_0\), with multiplicity \(k\), and that \(g\) has only one critical point \(s_0\), with the same multiplicity \(k\). Also, suppose that \(f(t_0)\) and \(g(s_0)\) have the same sign\footnote{Clearly, this is a necessary condition for \(f\) and \(g\) to be Lipschitz equivalent.}\textup{(}positive, negative or zero\textup{)}.
Then the following holds:
\begin{enumerate}[label=\roman*.]
\item If \(d\) is odd, then \(f\) and \(g\) are Lipschitz equivalent.
\item If \(d\) is even, then \(f\) and \(g\) are Lipschitz equivalent if and only 
if \/\(t_0\) and \(s_0\) are either both absolute minimum points, or both absolute maximum points  of \/\(f\) and \(g\), respectively.  
\end{enumerate}
\end{theoremA}

\begin{proof}
First, consider the case where \(d\) is odd.
If a real polynomial function of a single variable, of odd degree, has only one critical point, then it is a homeomorphism. Thus, under the assumption that 
\(d\) is odd, \(f\) and \(g\) are homeomorphisms. Choose a constant \(c > 0\) such that \(g(s_0) = c f(t_0)\), and define \(\phi \coloneqq g^{-1}\circ \hat{f}\from\R\to\R\), where \(\hat{f} \coloneqq c f\). The function \(\phi\) is a bijection such that \(g\circ\phi = c f\), and the multiplicity of \(f\) at \(t\) is equal to the multiplicity of \(g\) at \(\phi(t)\) for all \(t\in\R\). By Lemma \ref{lemma: bi-Lipschitz iff bi-analytic}, \(\phi\) is bi-Lipschitz.

Now, suppose that \(d\) is even. If a  real polynomial function of a single variable, of even degree, has only one critical point, then this critical point is a point of absolute extremum. Applying this to \(f\) and \(g\), we see that \(t_0\) is a point of absolute extremum for \(f\) and \(s_0\) is a point of absolute extremum for \(g\). If \(f\) and \(g\) are Lipschitz equivalent, then \(t_0\) and \(s_0\) are either both minimum points or both maximum points of \(f\) and \(g\), respectively. Otherwise, we would have \(\hat{f}(\R)\cap g(\R) = \{\hat{f}(t_0)\} = \{g(s_0)\}\), which is absurd, since the equation \(g\circ\phi = c f\) implies that 
\(\hat{f}(\R) = g(\R) \). Conversely, suppose that \(t_0\) and \(s_0\) are either both absolute minimum points or both absolute maximum points of \(f\) and \(g\), respectively.
Pick any constant \(c > 0\) for which \(g(s_0) = c f(t_0)\), and let
\(\phi\from\R\to\R\) be the function defined by
\(\phi |_{(-\infty,t_0]} \coloneqq 
     \left(g |_{(-\infty,s_0]} \right)^{-1}\circ \hat{f}|_{(-\infty,t_0]}\)
and
\(\phi |_{[t_0,+\infty)} \coloneqq 
     \left(g |_{[s_0,+\infty)} \right)^{-1}\circ \hat{f}|_{[t_0,+\infty)}\).
Clearly, \(\phi\) is a bijection such that \(g\circ \phi = c f\), and the multiplicity of \(f\) at \(t\) is equal to the multiplicity of \(g\) at \(\phi(t)\) for all \(t\in\R\). Hence, by Lemma \ref{lemma: bi-Lipschitz iff bi-analytic}, \(\phi\) is bi-Lipschitz.
\end{proof}

To address the case where \(f\) and \(g\) have the same number \(p\geq 2\) of critical points, we introduce a modified version of the {\it multiplicity symbol} defined in \cite{BFP}. Let \(f\from\R\to\R\) be a polynomial function of degree \(d\geq 1\), having exactly \(p\) critical points, with \(p\geq 2\). Let \(t_1< \cdots < t_p\) be the critical points of \(f\), with multiplicities \(\mu_1,\ldots,\mu_p\), respectively. The {\it multiplicity symbol} of \(f\) is the ordered pair \(\left(a,\mu\right)\) whose first entry is the \(p\)-tuple \(a = \left(f(t_1),\ldots,f(t_p)\right)\), and second entry is the \(p\)-tuple \(\mu = \left(\mu_1,\ldots,\mu_p\right)\). Let \(g\from\R\to\R\) be another polynomial function of degree \(d\geq 1\), having exactly the same number \(p\geq 2\) of critical points. Let \(s_1< \cdots < s_p\) be the critical points of \(g\), with multiplicities \(\nu_1,\ldots,\nu_p\), respectively. The multiplicity symbol of \(g\) is the ordered pair \(\left(b,\nu\right)\), where 
\(b = \left(g(s_1),\ldots,g(s_p)\right)\) and \(\nu = \left(\nu_1,\ldots,\nu_p\right)\). The multiplicity symbols \(\left(a,\mu\right)\) and \(\left(b,\nu\right)\) are said to be:
\begin{enumerate}[label=\roman*.]
\item {\it directly similar}, if there exists a constant \(c > 0\) such that \(b = c\cdot a\), and \(\nu = \mu\);
\item  {\it reversely similar}, if there exists a constant \(c > 0\) such that 
\(b = c\cdot \overline{a}\), and \(\nu=\overline{\mu}\).
\end{enumerate}
For any \(p\)-tuple \(x = \left(x_1,\ldots,x_p\right)\), 
\(\overline{x} := \left(x_p,\ldots,x_1\right)\) is the \(p\)-tuple \(x\) written in reverse order. The multiplicity symbols \(\left(a,\mu\right)\) and \(\left(b,\nu\right)\) are said to be {\it similar} if they are either directly similar or reversely similar.

\begin{theoremA}
\label{thm: at least 2 critical points}
 Let \(f,g\from\R\to\R\) be polynomial functions of the same degree 
 \(d\geq 1\), having the same number \(p\geq 2\) of critical points. Then \(f\) and \(g\) are Lipschitz equivalent if and only if their multiplicity symbols are similar.
\end{theoremA}

\begin{proof}
First, suppose that \(f\) and \(g\) are Lipschitz equivalent. Then there exist a bi-Lipschitz function \(\phi\from\R\to\R\) and a constant \(c > 0\) such that \(g\circ \phi = cf\). Since \(\phi\)\ preserves multiplicities (by Lemma \ref{lemma: bi-Lipschitz iff bi-analytic}), it follows that the multiplicity symbols of \(f\) and \(g\) are similar. (They are directly similar if \(\phi\) is increasing, and reversely similar if \(\phi\) is decreasing.)

Conversely, suppose that the multiplicity symbols of \(f\) and \(g\) are similar. We consider only the case where the multiplicity symbols of \(f\) and \(g\) are directly similar. The case where the multiplicity symbols of \(f\) and \(g\) are reversely similar is analogous. Let \(t_1< \cdots < t_p\) be the critical points of \(f\), with multiplicities \(\mu_1,\ldots,\mu_p\), respectively; and let \(s_1< \cdots < s_p\) be the critical points of \(g\), with multiplicities \(\nu_1,\ldots,\nu_p\), respectively. Since we are assuming that the multiplicity symbols of \(f\) and \(g\) are directly similar, there exists a constant \(c > 0\) such that \(g(s_i) = c f(t_i)\), for \(i =1,\ldots,p\); and \(\mu_i = \nu_i\), for \(i =1,\ldots,p\). 
Let \(\hat{f}\coloneqq cf\). The functions 
\(\hat{f} |_{[t_i,t_{i+1}]}\) and \(g |_{[s_i,s_{i+1}]}\) 
are both homeomorphisms onto their images, for \(1\leq i < p\).
Also, it is clear that \(g([s_i,s_{i+1}]) = \hat{f}([t_i,t_{i+1}])\).
Likewise, the functions 
\(\hat{f} |_{(-\infty,t_1]}\) and \(g |_{(-\infty,s_1]}\)
are homeomorphisms onto their images,
and so are the functions 
\(\hat{f} |_{[t_p,+\infty)}\) and \(g |_{[s_p,+\infty)}\).
Moreover, we also have \(\hat{f}((-\infty,t_1]) = g((-\infty,s_1])\) and \(\hat{f}([t_p,+\infty)) = g([s_p,+\infty))\). So, we define \(\phi\from\R\to\R\) by
\begin{align*}
\phi |_{(-\infty,t_1]} &\coloneqq 
     \left(g |_{(-\infty,s_1]}\right)^{-1}\circ \hat{f} |_{(-\infty,t_1]}\\[0.2cm]
\phi |_{[t_i,t_{i+1}]} &\coloneqq 
     \left(g |_{[s_i,s_{i+1}]}\right)^{-1}\circ \hat{f} |_{[t_i,t_{i+1}]}, 
     \text{ for } 1\leq i < p\\[0.2cm]
\phi |_{ [t_p,+\infty)} &\coloneqq 
     \left(g |_{[s_p,+\infty)}\right)^{-1}\circ \hat{f} |_{[t_p,+\infty)}.
\end{align*}
Clearly, \(\phi\) is a bijection such that \(g\circ\phi = c f\), and
it takes \(t_i\) to \(s_i\), for \(i=1,\ldots,p\). 
Since the multiplicity symbols of \(f\) and \(g\) are directly similar, it follows that the multiplicity of \(f\) at \(t\) is equal to the multiplicity of \(g\) at \(\phi(t)\) for all \(t\in\R\). Hence, by Lemma \ref{lemma: bi-Lipschitz iff bi-analytic}, \(\phi\) is bi-Lipschitz.
\end{proof}

\subsection{On the bi-Lipschitz transformation \(\phi\)}
\label{subsection: On the bi-Lipschitz transformation phi}

In this subsection, we investigate the asymptotic behavior of the bi-Lipschitz transformation \(\phi\) relating any two given nonconstant polynomial functions \(f,g\from\R\to\R\) that are Lipschitz equivalent. We begin with an auxiliary result establishing that the function \(\phi(t)/t\), defined on \(\R^*\coloneqq \R\setminus\{0\}\), can be extended to an analytic function on \(\R^*\cup\{\infty\}\subseteq {\mathbb P}^1(\R)\) with a nonzero real value at \(\infty\).

\begin{lemma}
\label{lemma: analytic extension of phi(t)/t}
Let \(\phi\from\R\to\R\) be a bi-Lipschitz homeomorphism that satisfies an equation of the form \(g\circ\phi = c f\), where \(f,g\from\R\to\R\) are nonconstant polynomial functions of the same degree, and \(c > 0\) is a constant. Then the following holds:
\begin{enumerate}[label = \roman*.]
\item \(\lim_{t\to+\infty} \phi(t)/t = \lim_{t\to-\infty} \phi(t)/t = \lambda\), where 
\(\lambda\) is a nonzero real number.
\item The function \(\psi\from\R\to\R\) defined by
\begin{equation}
\psi(t) \coloneqq
\begin{cases}
t\phi(t^{-1}), &\text{if }\,\,t\in\R\setminus\{0\}\\
\lambda, & \text{if } \,\,t = 0
\end{cases}
\end{equation}
is analytic.\footnote{This fact and the approach taken here to prove it were suggested to me by Dr. Maria Michalska.}
\end{enumerate}
\end{lemma}

\begin{proof}\ 
\begin{enumerate}[label=\roman*.]
\item
Let \(f(t) = \sum_{i = 0}^d a_i t^i\) and \(g(t) = \sum_{i= 0}^d b_i t^i\), where
\(a_d,b_d\neq 0\), and let
\begin{equation*}
l_+ \coloneqq \lim_{t\to+\infty} \frac{\phi(t)}{t}
\quad\text{and}\quad
l_- \coloneqq \lim_{t\to-\infty} \frac{\phi(t)}{t}.
\end{equation*}
Both of these limits are well-defined in the extended real line because \(\phi\) is semialgebraic (see Lemma \ref{lemma: semialgebraicity of phi}). Also, since \(\phi\) is bi-Lipschitz, \(l_+\) and \(l_-\) are nonzero real numbers.
Furthermore, since \(g\circ\phi = cf\), we have \(l_+^d = l_-^d = c\cdot a_d/b_d\), which implies that \(\abs{l_+} = \abs{l_-}\). On the other hand, since \(\phi\) is monotone and 
\(\lim_{\abs{t}\to+\infty} \abs{\phi(t)} = +\infty\), either \(\phi(t)\) and \(t\) have the same sign for large values of \(\abs{t}\), or \(\phi(t)\) and \(t\) have opposite signs for large values of \(\abs{t}\). In either case, \(l_+\) and \(l_-\) have the same sign, and therefore \(l_+ = l_-\).

\item
By Lemma \ref{lemma: bi-Lipschitz iff bi-analytic}, \(\psi\) is analytic at every point \(t\in\R\setminus\{0\}\), so we only need to prove that \(\psi\) is analytic at \(t = 0\). 
Let \(P(X,Y)\coloneqq g(Y) - c f(X)\) and let \(P^*(X,Y,Z)\) be the homogeneization of \(P\). 
Also, let \(f(t) = \sum_{i=0}^d a_it^i\) and \(g(t) = \sum_{i=0}^d b_it^i\), where 
\(a_d,b_d\neq 0\) and \(d\geq 1\). Thus,
\begin{equation*}
P(X,Y) = \sum_{i=0}^d b_iY^i - c\cdot\sum_{i=0}^d a_iX^i
\quad\text{and}\quad
P^*(X,Y,Z) = \sum_{i=0}^d b_iY^iZ^{d-i} - c\cdot\sum_{i=0}^d a_iX^iZ^{d-i}.
\end{equation*}

Since \(g(\phi(t)) = c f(t)\), we have \(P(t,\phi(t)) = 0\) for all \(t\in\R\).
Equivalently, \(P^*(t,\phi(t),1) = 0\) \,for all \(t\in\R\).
Since \(P^*\) is a homogeneous polynomial, it follows that 
\(P^*(1,\phi(t)/t, 1/t) = 0\) \,for all \(t\in\R\setminus\{0\}\).
Equivalently,
\begin{equation*}
P^*(1,t\phi(t^{-1}), t) = 0 \quad\text{for all } t\in\R\setminus\{0\}.
\end{equation*}

Let \(\widetilde{P}(Y,Z)\coloneqq P^*(1,Y,Z)\). From the computations above, it follows that
\begin{equation*}
\widetilde{P}(\psi(t),t) = 0\quad\text{for all }t\in\R.
\end{equation*}

Since \(\widetilde{P}(\lambda,0) = 0\) and 
\(\frac{\partial\widetilde{P}}{\partial y}(\lambda,0) = 
d\cdot b_d\cdot \lambda^{d-1}\neq 0\), the Implicit Function Theorem guarantees that there exists an analytic function \(\widetilde{\psi}\from I\to J\) from an open interval \(I\) containing \(0\) to an open interval \(J\) containing \(\lambda\) such that 
\begin{equation*}
\widetilde{P}(y,z) = 0 \iff y = \widetilde{\psi}(z),
\text{ for all }y\in J, z\in I.
\end{equation*}
Since \(\psi\) is continuous and \(\psi(0) = \lambda\), there exists an open interval \(I_0\subseteq I\) containing \(0\) such that \(\psi(t)\in J\) for all \(t\in I_0\). And since \(\widetilde{P}(\psi(t),t) = 0\) for all \(t\in I_0\), it follows that 
\(\psi(t) = \widetilde{\psi}(t)\) for all \(t\in I_0\). Hence, \(\psi\) is analytic at \(t = 0\).
\end{enumerate}
\end{proof}

\begin{proposition}
\label{prop: asymptotic formula for phi}
Let \(\phi\from\R\to\R\) be a bi-Lipschitz function that satisfies an equation of the form \(g\circ \phi = c f\), where \(f,g\from\R\to\R\) are nonconstant polynomial functions of the same degree, and \(c > 0\) is a constant. Then there exist \(\lambda,k\in\R\), with \(\lambda\neq 0\), such that
\begin{equation}
\label{eq: asymptotic formula for phi}
\phi(t) = (\lambda t + k) + \alpha(t),
\end{equation}
where \(\alpha\from\R\to\R\) is a Lipschitz analytic function such that:
\begin{enumerate}[label = \roman*.]
\item \(\lim_{\abs{t}\to+\infty} \alpha(t) = 0\)
\item \(\alpha\) can be analytically extended to 
\(\R\cup\{\infty\}\cong {\mathbb P}^1(\R)\).
\end{enumerate}
\end{proposition}

\begin{proof}
Let \(\lambda\in\R\setminus\{0\}\) and \(\psi\from\R\to\R\) be as in Lemma \ref{lemma: analytic extension of phi(t)/t}. We prove that 
\(\lim_{\abs{t}\to+\infty} \phi(t) - \lambda t\) is a well-defined real number.
For \(t\neq 0\), we have
\begin{equation*}
\phi(t) - \lambda t = t\psi(t^{-1}) - \lambda t = 
\frac{\psi(t^{-1}) - \lambda}{t^{-1}}.
\end{equation*}
Hence,
\begin{equation*}
\lim_{\abs{t}\to+\infty} \phi(t) - \lambda t  
    = \lim_{t\to 0}\phi(t^{-1}) - \lambda t^{-1}
    = \lim_{t\to 0}\frac{\psi(t) - \lambda}{t}
    = \psi^\prime(0)\in\R.\\[8pt]
\end{equation*}

Now, let \(k\coloneqq \lim_{\abs{t}\to+\infty} \phi(t) - \lambda t\).
Obviously, the only function \(\alpha\from\R\to\R\) satisfying 
(\ref{eq: asymptotic formula for phi}) is the one given by
\(\alpha(t)\coloneqq \phi(t) - (\lambda t + k)\),
which is a Lipschitz analytic function (because it is the difference of two Lipschitz analytic functions). From the definition of \(k\), it is immediate that 
\(\lim_{\abs{t}\to+\infty} \alpha(t) = 0\).
It remains to show that \(\alpha\) can be analytically extended to 
\(\R\cup\{\infty\}\cong {\mathbb P}^1(\R)\). 

For \(t\in\R\setminus\{0\}\), we have
\begin{align*}
\alpha(t^{-1}) &= \phi(t^{-1}) - (\lambda t^{-1} + k)\\[3pt]
&= (\psi(t) - \lambda)\cdot t^{-1} - k
\end{align*}
On the other hand, since \(\psi(0) = \lambda\) and \(\psi^{\prime}(0) = k\), we have
\begin{equation*}
\psi(t) = \lambda + kt + \sum_{k = 2}^{\infty} c_kt^k, 
\text{ for \(\abs{t}\) sufficiently small.}
\end{equation*}
Hence,
\begin{equation*}
\alpha(t^{-1}) = \sum_{k = 1}^{\infty} c_{k+1} t^k,
\text{ for \(\abs{t} > 0\) sufficiently small.}
\end{equation*}
Therefore, the function
\(\widehat{\alpha}\from\R\cup\{\infty\}\to\R\) defined by
\begin{equation*}
\widehat{\alpha}(t)\coloneqq
\begin{cases}
\alpha(t), &\text{if \,}t\in\R,\\
0, &\text{if \,}t = \infty
\end{cases}
\end{equation*}
is analytic.
\end{proof}

\section{\({\cal R}\)-Semialgebraic Lipschitz equivalence of \(\beta\)-qua\-si\-ho\-mo\-ge\-ne\-ous polynomials}
\label{section: R-semialg Lip equiv of beta-qh polynomials}

\begin{definition}
\label{def: beta-quasihomogeneous polynomial}
Let \(\beta\)\/ be a rational number greater than \(1\) and let \(d\)\/ be a positive integer. A polynomial \(F(X,Y)\in\R[X,Y]\) is said to be 
{\it \(\beta\)-quasihomogeneous of degree \(d\)}\/ if 
\begin{equation*}
F(tX,t^\beta Y) = t^d F(X,Y), \text{ for all }t > 0.
\end{equation*}
The positive integer \(d\)\/ is called the {\it \(\beta\)-quasihomogeneous degree} of \(F\).
\end{definition}

 \begin{remark}
According to Definition \ref{def: beta-quasihomogeneous polynomial}, the zero polynomial is \(\beta\)-quasihomogeneous of degree \(d\), for all rational numbers \(\beta > 1\) and all positive integers \(d\). However, throughout this paper, the term \(\beta\)-quasihomogeneous polynomial will always refer to a nonzero \(\beta\)-quasihomogeneous polynomial.
\end{remark}

\begin{remark}
If \(\beta = r/s\), where \(r > s > 0\) and \(\gcd (r,s) = 1\), then
the \(\beta\)-quasihomogeneous polynomials of degree \(d\) are those of the form 
\begin{equation*}
F(X,Y) = \sum_{k = 0}^m c_k X^{d - rk}Y^{sk},
\end{equation*}
where the coefficients \(c_k\) are real numbers, \(c_m\neq 0 \), and 
\(0\leq m\leq \lfloor d/r\rfloor\).
\end{remark} 

In this section, we address the problem of determining whether any two given 
\(\beta\)-quasihomogeneous polynomials are \({\cal R}\)-semialgebraically Lipschitz equivalent. 

\begin{definition}
Two \(\beta\)-quasihomogeneous polynomials
\(F(X,Y)\) and \(G(X,Y)\) are said to be {\it \({\cal R}\)-semialgebraically Lipschitz equivalent}\/ if there exists a germ of semialgebraic bi-Lipschitz homeomorphism \(\Phi\from(\R^2,0)\to(\R^2,0)\) such that \(G\circ\Phi = F\).
\end{definition}

Following the approach taken in \cite{BFP}, we reduce this problem to the simpler one of determining whether two given polynomial functions of a single variable are Lipschitz equivalent. In order to do this, we associate with each \(\beta\)-quasihomogeneous polynomial \(F(X,Y)\) two polynomial functions of a single variable, called its {\it height functions}.

\begin{definition}
Given a \(\beta\)-quasihomogeneous polynomial \(F(X,Y)\), the right height function and the left height function of \(F(X,Y)\) are defined to be, respectively, the functions \(f_+\from\R\to\R\), given by \(f_+(t) \coloneqq F(1,t)\), and \(f_-\from\R\to\R\), given by \(f_-(t) \coloneqq F(-1,t)\).\end{definition}

\begin{remark}
For any \(\beta\)-quasihomogeneous polynomial \(F(X,Y)\), we have:
\begin{equation*}
F(x,t\abs{x}^\beta) = 
\begin{cases}
\abs{x}^d f_+(t), &\text{if }\, x > 0,\\
\abs{x}^d f_-(t), &\text{if }\, x < 0.
\end{cases}
\end{equation*}
Also, note that \(\{(x,t\abs{x}^\beta): x, t\in\R\} = \R^2\).
\end{remark}

In this section, we establish conditions under which questions \ref{question 1} and \ref{question 2}, stated in Section \ref{section: intro}, are answered affirmatively. Thus, under appropriate conditions, we reduce the problem of determining whether two given \(\beta\)-quasihomogeneous polynomials are \({\cal R}\)-semialgebraically Lipschitz equivalent to determining whether their height functions can be arranged into pairs of Lipschitz equivalent functions. Since the problem of determining whether two given polynomial functions of a single variable are Lipschitz equivalent was already solved in Section \ref{subsection: Lipschitz equivalence of polynomial functions of a single variable}, this enables us to determine, under suitable conditions, whether two given \(\beta\)-quasihomogeneous polynomials are  
\({\cal R}\)-semialgebraically Lipschitz equivalent.

\subsection{\(\beta\)-regular germs and the \(\beta\)-transform}
\label{subsection: beta-isomorphisms and the beta-transform}

In this subsection, we provide conditions under which the \({\cal R}\)-semialgebraic Lipschitz equivalence of two \(\beta\)-quasihomogeneous polynomials of degree \(d\) implies that their height functions can be arranged into pairs of Lipschitz equivalent functions (see Theorem \ref{thm: equivalent polynomials, Lipschitz equivalent height functions} at the end of this section).

\begin{definition}
A germ of semialgebraic bi-Lipschitz homeomorphism \(\Phi = (\Phi_1,\Phi_2)\from (\R^2,0)\to(\R^2,0)\) is said to be \(\beta\)-regular if it satisfies the following conditions:
\begin{enumerate}[label = \roman*.]
\item \(\lim_{x\to 0^+}\Phi_1(x,0)/x\neq 0\) \ and 
\ \(\lim_{x\to 0^-}\Phi_1(x,0)/x\neq 0\).
\item For each \(t\in\R\), \(\Phi_2(x,t\abs{x}^\beta) = O(\abs{x}^\beta)\) as \(x\to 0\).
\end{enumerate}
\end{definition}

\begin{remark}
The limits in condition \((i)\) above are finite because \(\Phi\) is Lipschitz.
\end{remark}

\begin{remark}
\label{rk: equality of initial velocities}
Let \(\Phi\from (\R^2,0)\to(\R^2,0)\) be any germ of semialgebraic bi-Lipschitz map. Since semialgebraic Lipschitz maps transform paths with the same initial velocity into paths with the same initial velocity, it follows that for all \(t\in\R\), we have:
\begin{enumerate}[label=\roman*.]
\item The initial velocity of the path 
\(\tilde\gamma_+(x) = \Phi(x,tx^\beta)\), \(0\leq x < \epsilon\),
is equal to the initial velocity of the path
\(\tilde\alpha_+(x) = \Phi(x,0)\), \(0\leq x <\epsilon\).
\item The initial velocity of the path 
\(\tilde\gamma_-(x) = \Phi(-x,tx^\beta)\), \(0\leq x < \epsilon\),
is equal to the initial velocity of the path
\(\tilde\alpha_-(x) = \Phi(-x,0)\), \(0\leq x <\epsilon\).
\end{enumerate}
Hence, for all \(t\in\R\), 
\begin{equation}
\label{eq: equality of first components of initial velocities}
\lim_{x\to 0^+}\frac{\Phi_1(x,tx^\beta)}{x} = 
\lim_{x\to 0^+}\frac{\Phi_1(x,0)}{x}
\quad\text{and}\quad
\lim_{x\to 0^+}\frac{\Phi_1(-x,tx^\beta)}{x} = 
\lim_{x\to 0^+}\frac{\Phi_1(-x,0)}{x}\,.
\end{equation}
Also, for all \(t\in\R\), 
\begin{equation}
\label{eq: equality of second components of initial velocities}
\lim_{x\to 0^+}\frac{\Phi_2(x,tx^\beta)}{x} = 
\lim_{x\to 0^+}\frac{\Phi_2(x,0)}{x}
\quad\text{and}\quad
\lim_{x\to 0^+}\frac{\Phi_2(-x,tx^\beta)}{x} = 
\lim_{x\to 0^+}\frac{\Phi_2(-x,0)}{x}\,.
\end{equation}
\end{remark}

\begin{proposition}
\label{prop: opposite horizontal initial velocities (beta-isomorphism)}
If\/ \(\Phi\from (\R^2,0)\to(\R^2,0)\) is a \(\beta\)-regular germ of semialgebraic bi-Lipschitz homemomorphism then the initial velocities of the paths \(\tilde\alpha_+(x) = \Phi(x,0)\), \(0\leq x < \epsilon\), and \(\tilde\alpha_-(x) = \Phi(-x,0)\), \(0\leq x < \epsilon\), are horizontal\/\footnote{We say that a vector 
\((v_1,v_2)\in\R^2\) is {\it horizontal} if \(v_1\neq 0\) and \(v_2 = 0\).} and have opposite directions.
\end{proposition}

\begin{proof}
Let \(\Phi\) be a \(\beta\)-regular germ of semialgebraic bi-Lipschitz homemomorphism. Then, in particular, we have \(\Phi_2(x,0) = O(\abs{x}^\beta)\) as \(x\to 0\). Consequently, 
\(\lim_{x\to 0^+}\Phi_2(x,0)/x = \lim_{x\to 0^+}\Phi_2(-x,0)/x = 0\). Since the paths \(\tilde\alpha_+(x) = \Phi(x,0)\), \(0\leq x <\epsilon\), and \(\tilde\alpha_-(x) = \Phi(-x,0)\), \(0\leq x <\epsilon\), both have nonzero finite initial velocity (because \(\Phi\) is bi-Lipschitz), it follows that both of them have horizontal initial velocity. Furthermore, the initial velocities \(\tilde\alpha_+^\prime(0)\) and \(\tilde\alpha_-^\prime(0)\) do not have the same direction because the initial velocities of the paths \(\alpha_+(x) = (x,0)\), \(0\leq x <\epsilon\), and \(\alpha_-(x) = (-x,0)\), \(0\leq x <\epsilon\), do not have the same direction and \(\Phi\) is bi-Lipschitz. Hence the result.
\end{proof}

It follows from Proposition \ref{prop: opposite horizontal initial velocities (beta-isomorphism)} that each \(\beta\)-regular germ of semialgebraic bi-Lipschitz homeomorphism \(\Phi = (\Phi_1,\Phi_2)\from(\R^2,0)\to(\R^2,0)\) satisfies exactly one of the following conditions:
\begin{enumerate}[label = (\Alph*)]
\item \label{condition: direct beta-isomorphism}
\(\lim_{x\to 0^+} \Phi_1(x,0)/x > 0\)
\,and \,\,\(\lim_{x\to 0^-} \Phi_1(x,0)/x > 0\);
\item \label{condition: reverse beta-isomorphism}
 \(\lim_{x\to 0^+} \Phi_1(x,0)/x < 0\)
\,and \,\,\(\lim_{x\to 0^-} \Phi_1(x,0)/x < 0\).
\end{enumerate}

\begin{definition}
A \(\beta\)-regular germ of semialgebraic bi-Lipschitz homeomorphism \(\Phi\from(\R^2,0)\to(\R^2,0)\) is said to be {\it direct} if it satisfies 
\ref{condition: direct beta-isomorphism}, and it is said to be {\it reverse}
if it satisfies \ref{condition: reverse beta-isomorphism}.
\end{definition}

\begin{definition}
Given a \(\beta\)-regular germ of semialgebraic bi-Lipschitz homeomorphism 
\(\Phi = (\Phi_1,\Phi_2)\from (\R^2,0) \to (\R^2,0)\),
we define
\(\lambda \coloneqq (\lambda_+,\lambda_-)\) 
and 
\(\phi \coloneqq (\phi_+,\phi_-)\),
where:
\begin{align*}
\lambda_+ \coloneqq \lim_{x\to 0^+}\frac{\Phi_1(x,0)}{x}\,,& \quad
\lambda_- \coloneqq \lim_{x\to 0^-}\frac{\Phi_1(x,0)}{x}\\[7pt]
\phi_+(t) \coloneqq \abs{\lambda_+}^{-\beta}\cdot
          \lim_{x\to 0^+}\frac{\Phi_2(x,t\abs{x}^\beta)}{\abs{x}^\beta}\,,&\quad
\phi_-(t) \coloneqq \abs{\lambda_-}^{-\beta}\cdot
          \lim_{x\to 0^-}\frac{\Phi_2(x,t\abs{x}^\beta)}{\abs{x}^\beta}          
\end{align*}
The ordered pair \((\lambda,\phi)\) is called the \(\beta\)-transform of \(\Phi\). 
\end{definition}

From the definition above, it follows that, for each \(t\in\R\):
\begin{align}
\label{eq: Phi1 asymptotic formula}
\Phi_1(x,t\abs{x}^\beta) &= 
          \lambda x + o(x) \quad \text{as } x\to 0\\
\label{eq: Phi2 asymptotic formula}
\Phi_2(x,t\abs{x}^\beta) &= 
          \abs{\lambda}^\beta\phi(t)\abs{x}^\beta + o(\abs{x}^\beta) 
          \quad \text{as } x\to 0          
\end{align}
where
\begin{equation*}
\begin{cases}
\lambda = \lambda_+ \,\text{ and } \,\phi = \phi_+, &\text{if } \,x > 0,\\
\lambda = \lambda_- \,\text{ and } \,\phi = \phi_-, &\text{if } \,x < 0.
\end{cases}
\end{equation*}

\begin{proposition}
Let \(\Phi = (\Phi_1,\Phi_2)\from (\R^2,0)\to(\R^2,0)\) and 
\(\Psi = (\Psi_1,\Psi_2)\from (\R^2,0)\to(\R^2,0)\) be 
\(\beta\)-regular germs of semialgebraic bi-Lipschitz homeomorphisms, and let \(Z \coloneqq \Psi\circ\Phi\).
Let \((\lambda,\phi)\), \((\mu,\psi)\) be the 
\(\beta\)-transforms of\/ \(\Phi\), \(\Psi\), respectively. Then the following holds:
\begin{enumerate}[label=\roman*.]
\item Asymptotic formula for \(Z_1(x,0)\).
\begin{equation}
\label{eq: asymptotic formula for Z1}
Z_1(x,0) = \lambda\mu x + o(x) \,\text{ as } x\to 0,
\end{equation}
where 
\begin{equation*}
\begin{cases}
\lambda = \lambda_+,\, \mu = \mu_+, & \text{if } \,x > 0,\\
\lambda = \lambda_-,\, \mu = \mu_-, & \text{if } \,x < 0\\
\end{cases}
\quad\text{ or }\quad
\begin{cases}
\lambda = \lambda_+,\, \mu = \mu_-, & \text{if } \,x > 0,\\
\lambda = \lambda_-,\, \mu = \mu_+, & \text{if } \,x < 0,\\
\end{cases}
\end{equation*}
depending on whether \(\Phi\) is direct or reverse, respectively.

\item Asymptotic formula for \(Z_2(x,t\abs{x}^\beta)\), where \(t\) is fixed. 
\begin{equation}
\label{eq: asymptotic formula for Z2}
Z_2(x,t\abs{x}^\beta) = 
\abs{\lambda\mu}^\beta\psi(\phi(t))\abs{x}^\beta + o(\abs{x}^\beta) 
\,\text{ as } x\to 0,
\end{equation} 
where
\begin{equation*}
\begin{cases}
\lambda = \lambda_+,\, \mu = \mu_+,\, \phi = \phi_+,\, \psi = \psi_+, 
& \text{if } \,x > 0,\\
\lambda = \lambda_-,\, \mu = \mu_-,\, \phi = \phi_-,\, \psi = \psi_-, 
& \text{if } \,x < 0\\
\end{cases}
\end{equation*}
\quad\text{ or }\quad
\begin{equation*}
\begin{cases}
\lambda = \lambda_+,\, \mu = \mu_-,\, \phi = \phi_+,\, \psi = \psi_-, 
& \text{if } \,x > 0,\\
\lambda = \lambda_-,\, \mu = \mu_+,\, \phi = \phi_-,\, \psi = \psi_+, 
& \text{if } \,x < 0,\\
\end{cases}
\end{equation*}
depending on whether \(\Phi\) is direct or reverse, respectively.
\end{enumerate}
\end{proposition} 

\begin{proof}
By definition,
\begin{equation*}
Z_1(x,0) = \Psi_1(\Phi(x,0)).
\end{equation*}
Then, by (\ref{eq: Phi1 asymptotic formula}) and (\ref{eq: Phi2 asymptotic formula}),
\begin{equation*}
Z_1(x,0) = \Psi_1(\lambda x, \abs{\lambda}^\beta\phi(0)\abs{x}^\beta) + o(x),
\end{equation*}
where
\begin{equation}
\label{eq: lambda and phi - first time}
\begin{cases}
\lambda = \lambda_+ \,\text{ and } \,\phi = \phi_+, &\text{if } \,x > 0,\\
\lambda = \lambda_- \,\text{ and } \,\phi = \phi_-, &\text{if } \,x < 0.
\end{cases}
\end{equation}
And then, by applying (\ref{eq: Phi1 asymptotic formula}) to \(\Psi_1\) in the equation above, we obtain:
\begin{equation*}
Z_1(x,0) = \lambda\mu x + o(x),
\end{equation*}
where
\begin{equation}
\label{eq: mu}
\begin{cases}
\mu = \mu_+, &\text{if } \,\lambda x > 0,\\
\mu = \mu_-, &\text{if } \,\lambda x < 0.\\
\end{cases}
\end{equation}
Since
\begin{equation*}
\begin{cases}
\lambda_+ > 0 \text{ and } \lambda_- > 0, &\text{if \(\Phi\) is direct}, \\
\lambda_+ < 0 \text{ and } \lambda_- < 0, &\text{if \(\Phi\) is reverse}, \\
\end{cases}
\end{equation*}
it follows from (\ref{eq: lambda and phi - first time}) and (\ref{eq: mu}) that
\begin{equation*}
\begin{cases}
\lambda = \lambda_+,\, \mu = \mu_+, & \text{if } \,x > 0,\\
\lambda = \lambda_-,\, \mu = \mu_-, & \text{if } \,x < 0\\
\end{cases}
\quad\text{ or }\quad
\begin{cases}
\lambda = \lambda_+,\, \mu = \mu_-, & \text{if } \,x > 0,\\
\lambda = \lambda_-,\, \mu = \mu_+, & \text{if } \,x < 0,\\
\end{cases}
\end{equation*}
depending on whether \(\Phi\) is direct or reverse, respectively.

Now, let \(t\in\R\) be fixed. By definition,
\begin{equation}
\label{eq: Z_2, part 1}
Z_2(x,t\abs{x}^\beta) = \Psi_2(\tilde x, \Phi_2(x,t\abs{x}^\beta)),
\end{equation}
where \(\tilde x \coloneqq \Phi_1(x,t\abs{x}^\beta)\).
By (\ref{eq: Phi2 asymptotic formula}), we have
\begin{equation}
\label{eq: Phi2 asymptotic formula slightly modified}
\Phi_2(x,t\abs{x}^\beta) = 
          \phi(t)\abs{\lambda x}^\beta + o(\abs{x}^\beta),
\end{equation}
where
\begin{equation}
\label{eq: lambda and phi - second time}
\begin{cases}
\lambda = \lambda_+ \,\text{ and } \,\phi = \phi_+, &\text{if } \,x > 0,\\
\lambda = \lambda_- \,\text{ and } \,\phi = \phi_-, &\text{if } \,x < 0.
\end{cases}
\end{equation}
By (\ref{eq: Phi1 asymptotic formula}), we have
\begin{equation*}
\lim_{x\to 0}\frac{\tilde{x}}{\lambda x} = 1,
\end{equation*}
so that
\begin{equation}
\label{eq: |tilde x|^beta/|lambda x|^beta -> 1}
\lim_{x\to 0}\frac{\abs{\tilde{x}}^\beta}{\abs{\lambda x}^\beta} = 1
\end{equation}
and hence 
\begin{equation}
\label{eq: |lambda x|^beta is comparable to |x|^beta}
\abs{\lambda x}^\beta = \abs{\tilde{x}}^\beta + o(\abs{x}^\beta).
\end{equation}
From (\ref{eq: Phi2 asymptotic formula slightly modified}) and (\ref{eq: |lambda x|^beta is comparable to |x|^beta}), we get
\begin{equation*}
\Phi_2(x,t\abs{x}^\beta) = 
          \phi(t)\abs{\tilde{x}}^\beta + o(\abs{x}^\beta).
\end{equation*} 
Thus,
\begin{equation}
\label{eq: Z_2, part 2}
\Psi_2(\tilde{x},\Phi_2(x,t\abs{x}^\beta)) = 
\Psi_2(\tilde{x}, \phi(t)\abs{\tilde{x}}^\beta) + o(\abs{x}^\beta).
\end{equation} 
By applying (\ref{eq: Phi2 asymptotic formula}) to \(\Psi_2\) and using 
(\ref{eq: |tilde x|^beta/|lambda x|^beta -> 1}), we obtain
\begin{equation*}
\Psi_2(\tilde{x}, \phi(t)\abs{\tilde{x}}^\beta) = 
   \abs{\mu}^\beta\psi(\phi(t))\abs{\tilde{x}}^\beta + o(\abs{x}^\beta),
\end{equation*}
where
\begin{equation}
\label{eq: mu and psi}
\begin{cases}
\mu = \mu_+ \,\text{ and } \,\psi = \psi_+, &\text{if } \,\tilde{x} > 0,\\
\mu = \mu_- \,\text{ and } \,\psi = \psi_-, &\text{if } \,\tilde{x} < 0.
\end{cases}
\end{equation}
By (\ref{eq: |lambda x|^beta is comparable to |x|^beta}), it follows that
\begin{equation}
\label{eq: Z_2, part 3}
\Psi_2(\tilde{x}, \phi(t)\abs{\tilde{x}}^\beta) = 
   \abs{\lambda\mu}^\beta\psi(\phi(t))\abs{x}^\beta + o(\abs{x}^\beta).
\end{equation}
Then, by (\ref{eq: Z_2, part 1}), (\ref{eq: Z_2, part 2}), and (\ref{eq: Z_2, part 3}), we have
\begin{equation*}
Z_2(x,t\abs{x}^\beta) = \abs{\lambda\mu}^\beta\psi(\phi(t))\abs{x}^\beta + o(\abs{x}^\beta).
\end{equation*}

Since \(\tilde{x}\) and \(\lambda x\) have the same sign for small values of \(x\), it follows from (\ref{eq: mu and psi}) that
\begin{equation}
\label{eq: mu and psi - modified}
\begin{cases}
\mu = \mu_+ \,\text{ and } \,\psi = \psi_+, &\text{if } \,\lambda x > 0,\\
\mu = \mu_- \,\text{ and } \,\psi = \psi_-, &\text{if } \,\lambda x < 0.
\end{cases}
\end{equation}
And since 
\begin{equation*}
\begin{cases}
\lambda_+ > 0 \text{ and } \lambda_- > 0, &\text{if \(\Phi\) is direct}, \\
\lambda_+ < 0 \text{ and } \lambda_- < 0, &\text{if \(\Phi\) is reverse}, \\
\end{cases}
\end{equation*}
it follows from (\ref{eq: lambda and phi - second time}) and (\ref{eq: mu and psi - modified}) that
\begin{equation*}
\begin{cases}
\lambda = \lambda_+,\, \mu = \mu_+,\, \phi = \phi_+,\, \psi = \psi_+, 
& \text{if } \,x > 0,\\
\lambda = \lambda_-,\, \mu = \mu_-,\, \phi = \phi_-,\, \psi = \psi_-, 
& \text{if } \,x < 0\\
\end{cases}
\end{equation*}
\quad\text{ or }\quad
\begin{equation*}
\begin{cases}
\lambda = \lambda_+,\, \mu = \mu_-,\, \phi = \phi_+,\, \psi = \psi_-, 
& \text{if } \,x > 0,\\
\lambda = \lambda_-,\, \mu = \mu_+,\, \phi = \phi_-,\, \psi = \psi_+, 
& \text{if } \,x < 0,\\
\end{cases}
\end{equation*}
depending on whether \(\Phi\) is direct or reverse, respectively.
\end{proof}

\begin{corollary}
\label{cor: beta-transform of a composition}
Let \(\Phi = (\Phi_1,\Phi_2)\from (\R^2,0)\to(\R^2,0)\) and 
\(\Psi = (\Psi_1,\Psi_2)\from (\R^2,0)\to(\R^2,0)\) be 
\(\beta\)-regular germs of semialgebraic bi-Lipschitz homeomorphisms.
Then \(Z \coloneqq \Psi\circ\Phi\) is a \(\beta\)-regular germ of semialgebraic bi-Lipschitz homeomorphism. Furthermore, the following holds:
\begin{equation*}
(\nu,\zeta) = 
\begin{cases}
\left(
(\lambda_+\mu_+,\lambda_-\mu_-), (\psi_+\circ\phi_+,\psi_-\circ\phi_-)
\right), 
&\text{if\/ \(\Phi\) is direct,}\\
\left(
(\lambda_+\mu_-,\lambda_-\mu_+), (\psi_-\circ\phi_+,\psi_+\circ\phi_-)
\right), 
&\text{if\/ \(\Phi\) is reverse,}
\end{cases}
\end{equation*}
where \((\lambda,\phi)\), \((\mu,\psi)\), and \((\nu,\zeta)\) are the \(\beta\)-transforms of\/ \(\Phi\), \(\Psi\), and\/ \(Z\), respectively. 
\end{corollary}

\begin{corollary}
\label{cor: beta-transform of the inverse of a beta-isomorphism}
Let \(\Phi\from(\R^2,0)\to(\R^2,0)\) be a \(\beta\)-regular germ of semialgebraic bi-Lipschitz homeomorphism, 
and let \((\lambda,\phi)\) be the \(\beta\)-transform of \(\Phi\).
If\/ \(\Phi^{-1}\) is \(\beta\)-regular, then \(\lambda_+\), \(\lambda_-\), \(\phi_+\), and \(\phi_-\) are all invertible, and the \(\beta\)-transform of\/ \(\Phi^{-1}\) is given by:
\begin{equation*}
\begin{cases}
\left((\lambda_+^{-1},\lambda_-^{-1}),(\phi_+^{-1},\phi_-^{-1})\right),
&\text{if\/ \(\Phi\) is direct,}\\
\left((\lambda_-^{-1},\lambda_+^{-1}),(\phi_-^{-1},\phi_+^{-1})\right),
&\text{if\/ \(\Phi\) is reverse.}
\end{cases}
\end{equation*} 
\end{corollary}

\begin{proof}
We prove the result in the case where \(\Phi\) is direct; the analogous case where \(\Phi\) is reverse follows similarly. Assume that \(\Phi\) is a direct \(\beta\)-regular germ of semialgebraic bi-Lipschitz homeomorphism and that \(\Phi^{-1}\) is also \(\beta\)-regular. Let \((\mu,\psi) = ((\mu_+,\mu_-),(\psi_+,\psi_-))\) denote the \(\beta\)-transform of \(\Phi^{-1}\).  
We must show that \(\lambda_+\mu_+=1\), \(\lambda_-\mu_- = 1\), 
\(\psi_+\circ\phi_+ = \phi_+\circ\psi_+ = \id_\R\), and 
\(\psi_-\circ\phi_- = \phi_-\circ\psi_- = \id_\R\).

By Corollary \ref{cor: beta-transform of a composition}, the \(\beta\)-transform of \(\Phi^{-1}\circ\Phi\) is given by
\(((\lambda_+\mu_+,\lambda_-\mu_-),(\psi_+\circ\phi_+,\psi_-\circ\phi_-))\).
On the other hand, we have \(\Phi^{-1}\circ\Phi = I\), where \(I\from(\R^2,0)\to(\R^2,0)\) is the germ of the identity map on \(\R^2\). Thus, the \(\beta\)-transform of 
\(\Phi^{-1}\circ\Phi\) is equal to the \(\beta\)-transform of \(I\), which is
\(((1,1),(\id_\R,\id_\R))\).
Hence, \(\lambda_+ \mu_+ = \lambda_- \mu_- = 1\) and 
\(\psi_+\circ\phi_+ = \psi_-\circ\phi_- = \id_\R\).
Since \(\Phi\) is direct, we have \(\lambda_+ > 0\) and \(\lambda_- > 0\), which implies that \(\mu_+ > 0\) and \(\mu_- > 0\). Thus, \(\Phi^{-1}\) is also direct. Then, by Corollary \ref{cor: beta-transform of a composition}, the \(\beta\)-transform of \(\Phi\circ\Phi^{-1}\) is given by
\(((\lambda_+\mu_+,\lambda_-\mu_-),(\phi_+\circ\psi_+,\phi_-\circ\psi_-))\).
Finally, since \(\Phi\circ\Phi^{-1} = I\), it follows that
\(\phi_+\circ\psi_+ = \phi_-\circ\psi_- = \id_\R\).
\end{proof}

\begin{proposition}
\label{prop: about lambda and phi in a beta-transform}
Let \(\Phi\from(\R^2,0)\to(\R^2,0)\) be a \(\beta\)-regular germ of semialgebraic bi-Lipschitz homeomorphism such that \(\Phi^{-1}\) is also \(\beta\)-regular, and let \((\lambda,\phi)\) denote the \(\beta\)-transform of\/ \(\Phi\). Then:
\begin{enumerate}[label=\textnormal{(\alph*)}]
\item \(\lambda_+\) and \/ \(\lambda_-\) are nonzero real numbers and they have the same sign.
\item \(\phi_+\) and \/ \(\phi_-\) are bi-Lipschitz functions.
\end{enumerate}
\end{proposition}

\begin{proof}
As pointed out just after the proof of Proposition \ref{prop: opposite horizontal initial velocities (beta-isomorphism)}, 
\(\lambda_+ = \lim_{x\to 0^+}\Phi_1(x,0)/x\) and
\(\lambda_- = \lim_{x\to 0^-}\Phi_1(x,0)/x\) are nonzero real numbers which have the same sign. This establishes the first part of the proposition. Now, let us prove the second part.

By (\ref{eq: Phi2 asymptotic formula}), 
for any fixed \(t\) and \(t^\prime\), we have:
\begin{equation*}
\phi_+(t) - \phi_+(t^\prime) =
 \abs{\lambda_+}^{-\beta}\cdot
 \frac{\Phi_2(x,t\abs{x}^\beta) - \Phi_2(x,t^\prime\abs{x}^\beta)}{\abs{x}^\beta}
+ \frac{o(\abs{x}^\beta)}{\abs{x}^\beta}, \quad\text{as } \,x\to 0^+
\end{equation*}
and
\begin{equation*}
\phi_-(t) - \phi_-(t^\prime) =
 \abs{\lambda_-}^{-\beta}\cdot
 \frac{\Phi_2(x,t\abs{x}^\beta) - \Phi_2(x,t^\prime\abs{x}^\beta)}{\abs{x}^\beta}
+ \frac{o(\abs{x}^\beta)}{\abs{x}^\beta}, \quad\text{as } \,x\to 0^-.
\end{equation*}
On the other hand, since \(\Phi_2\) is Lipschitz, there exists \(K > 0\) 
(independent of \(x,t,t^\prime\)) such that
\begin{equation*}
\abs{\Phi_2(x,t\abs{x}^\beta) - \Phi_2(x,t^\prime\abs{x}^\beta)} \leq
     K\abs{t - t^\prime}\abs{x}^\beta.
\end{equation*}
Thus,
\begin{equation}
\label{eq: estimate for abs(phi_+(t) - phi_+(t^prime))}
\abs{\phi_+(t) - \phi_+(t^\prime)} \leq 
\abs{\lambda_+}^{-\beta}K\abs{t - t^\prime} 
     + \frac{o(\abs{x}^\beta)}{\abs{x}^\beta} \quad\text{as } \,x\to 0^+
\end{equation}
and
\begin{equation}
\label{eq: estimate for abs(phi_-(t) - phi_-(t^prime))}
\abs{\phi_-(t) - \phi_-(t^\prime)} \leq 
\abs{\lambda_-}^{-\beta}K\abs{t - t^\prime} 
     + \frac{o(\abs{x}^\beta)}{\abs{x}^\beta} \quad\text{as } \,x\to 0^-.
\end{equation}
By letting \(x\to 0^+\) in (\ref{eq: estimate for abs(phi_+(t) - phi_+(t^prime))}), and \(x\to 0^-\) in (\ref{eq: estimate for abs(phi_-(t) - phi_-(t^prime))}), we obtain:
\begin{equation*}
\abs{\phi_+(t) - \phi_+(t^\prime)} \leq 
     \abs{\lambda_+}^{-\beta}K\abs{t - t^\prime} 
\quad\text{ and }\quad
\abs{\phi_-(t) - \phi_-(t^\prime)} \leq 
     \abs{\lambda_-}^{-\beta}K\abs{t - t^\prime} \,.
\end{equation*}
Therefore, both \(\phi_+\) and \(\phi_-\) are Lipschitz functions.

Up to this point, our argument shows that, if \(\Phi\) is \(\beta\)-regular, then the functions \(\phi_+\) and \(\phi_-\) are both Lipschitz. 
Since \(\Phi^{-1}\) is also \(\beta\)-regular, by hypothesis, this conclusion can also be applied to \(\Phi^{-1}\). Then, by Corollary \ref{cor: beta-transform of the inverse of a beta-isomorphism}, it follows that the functions 
\(\phi_+^{-1}\) and \(\phi_-^{-1}\) are both Lipschitz too.
\end{proof}

Let \(F(X,Y)\) be a \(\beta\)-quasihomogeneous polynomial of degree \(d\), and let \(f_+\) and \(f_-\) be the height functions of \(F\). Then we have:
\begin{equation*}
F^{-1}(0) \cap {\cal H}_+ = \bigcup_{t \in f_+^{-1}(0)} \left\{ (x,tx^\beta) : x > 0\right\}
\;\text{ and }\;
F^{-1}(0) \cap {\cal H}_- = \bigcup_{t \in f_-^{-1}(0)} \left\{ (x,t\abs{x}^\beta) : x < 0\right\},
\end{equation*}
where \({\cal H}_+ = \{(x,y)\in\R^2: x > 0\}\) and\, \({\cal H}_- = \{(x,y)\in\R^2: x < 0\}\). Furthermore, \((0,0)\) is the only point of \(F^{-1}(0)\) on the line \(X = 0\), unless \(X\) is a factor of \(F\). In the latter case, \(F^{-1}(0)\) contains the entire line \(X = 0\).

Thus, if \(X\) is not a factor of \(F(X,Y)\), the connected components of \(F^{-1}(0)\setminus\{(0,0)\}\) are the sets \(R_t = \{(x,tx^\beta): x > 0\}\), where \(t\in f_+^{-1}(0)\), and the sets \(L_t = \{(x,t\abs{x}^\beta): x < 0\}\), where \(t\in f_-^{-1}(0)\). If \(X\) is a factor of \(F(X,Y)\), then the connected components of \(F^{-1}(0)\setminus\{(0,0)\}\) are the sets \(R_t\), \(L_t\), and the sets \(V_+ = \{(0,y): y > 0\}\) and \(V_- = \{(0,y): y < 0\}\). We call the sets \(R_t\) (with \(t\in f_+^{-1}(0)\)) and \(L_t\) (with \(t\in f_-^{-1}(0)\)) the {\it regular components of \(F^{-1}(0)\setminus\{(0,0)\}\)}. If \(X\) is not a factor of \(F(X,Y)\) then all connected components of \(F^{-1}(0)\setminus\{(0,0)\}\) are regular.

\begin{lemma}
\label{lemma: characterization of beta-isomorphism via branch matching}
Let \(F(X,Y)\) and \(G(X,Y)\) be \(\beta\)-quasihomogeneous polynomials of degree \(d\), with \(G\) not of the form \(cX^d\). Suppose that \(F\) and \(G\) are \({\cal R}\)-semialgebraically Lipschitz equivalent, and let  
\(\Phi=(\Phi_1,\Phi_2)\from(\R^2,0)\to(\R^2,0)\) denote a germ of semialgebraic bi-Lipschitz homeomorphism such that \(G\circ\Phi = F\). Then \(\Phi\) establishes a one-to-one correspondence between the connected components of \(F^{-1}(0)\setminus\{(0,0)\}\) and those of \({G^{-1}(0)\setminus\{(0,0)\}}\). If \(F^{-1}(0)\setminus\{(0,0)\}\) has both a regular component \(R_t\) and a regular component \(L_t\) that are matched to regular components of \(G^{-1}(0)\setminus\{(0,0)\}\) under this correspondence, then \(\Phi\) is \(\beta\)-regular. Furthermore, under these conditions, the following holds:
\begin{equation}
\label{eq: gophi = lambda f}
\begin{cases}
g_+\circ \phi_+ = \abs{\lambda_+}^{-d} f_+ 
\quad\text{and\/}\quad
g_-\circ \phi_- = \abs{\lambda_-}^{-d} f_-,
&\text{if\/ \(\Phi\) is direct},\\[2pt]
g_-\circ \phi_+ = \abs{\lambda_+}^{-d} f_+ 
\quad\text{and\/}\quad
g_+\circ \phi_- = \abs{\lambda_-}^{-d} f_-,
&\text{if\/ \(\Phi\) is reverse}.
\end{cases}
\end{equation}
Here, \(f_+,f_-\) are the height functions of\/ \(F\), \(g_+,g_-\) are the height functions of\/ \(G\), and \((\lambda,\phi)\) is the \(\beta\)-transform of\/ \(\Phi\).
\end{lemma}

\begin{proof}
Suppose there exist \(t_0\in f_+^{-1}(0)\) and \(t_1\in f_-^{-1}(0)\) such that \(\Phi\) matches \(R_{t_0}\) and \(L_{t_1}\) to regular components of \(G^{-1}(0)\setminus\{(0,0)\}\).  We begin the proof that \(\Phi\) is \(\beta\)-regular by showing that \(\lim_{x\to 0^+}\Phi_1(x,0)/x\neq 0\) \,and \,\(\lim_{x\to 0^-}\Phi_1(x,0)/x\neq 0\). Assume that \(R_{t_0}\) is matched to a regular component of \(G^{-1}(0)\setminus\{(0,0)\}\) in the right half-plane. Then there exists \(s_0\in g_+^{-1}(0)\) such that 
\(\Phi_2(x,t_0x^\beta) = s_0 [\Phi_1(x,t_0 x^\beta)]^\beta\), for \(x > 0\) sufficiently small. Since the path \(\tilde\gamma_+(x) = \Phi(x,t_0x^\beta)\), \(0\leq x < \epsilon\), has nonzero initial velocity, this implies that \(\lim_{x\to 0^+}\Phi_1(x,t_0x^\beta)/x\neq 0\), and hence \(\lim_{x\to 0^+}\Phi_1(x,0)/x\neq 0\). Similarly, assuming that \(R_{t_0}\) is matched to a regular component of \(G^{-1}(0)\setminus\{(0,0)\}\) in the left half-plane, we can show that there exists \(s_0\in g_-^{-1}(0)\) such that 
\(\Phi_2(x,t_0x^\beta) = s_0 \abs{\Phi_1(x,t_0 x^\beta)}^\beta\), for \(x > 0\) sufficiently small. Then, as in the previous case, we conclude that \(\lim_{x\to 0^+}\Phi_1(x,0)/x\neq 0\). From the assumption that \(L_{t_1}\) is matched to a regular component of \(G^{-1}(0)\setminus\{(0,0)\}\), an analogous argument shows that \(\lim_{x\to 0^-}\Phi_1(x,0)/x\neq 0\).

Next, to complete the proof of the \(\beta\)-regularity of \(\Phi\), we show that for each \(t\in\R\), \(\Phi_2(x,t\abs{x}^\beta) = O(\abs{x}^\beta)\) as \(x\to 0\).
Let \(\beta = r/s\), where \(r > s > 0\) and \(\gcd(r,s) = 1\). 
Then we have \(G(X,Y) = \sum_{k=0}^m c_k X^{d-rk}Y^{sk}\),
where the coefficients \(c_k\) are real numbers, \(c_m\neq 0\), and 
\(1\leq m\leq\lfloor{d/r}\rfloor\). By hypothesis, \(G(\Phi(x,y)) = F(x,y)\). Thus, for any \(t\in\R\) and \(x \neq 0\) sufficiently small, we have
\(G(\Phi(x,t \abs{x}^\beta)) = F(x,t \abs{x}^\beta)\).
Since the polynomials \(F\) and \(G\) are both \(\beta\)-quasihomogeneous of degree \(d\), this implies that
\begin{equation}
\label{eq: gophi = lambda f prior to the limit}
G\left(\frac{\Phi_1(x,t\abs{x}^\beta)}{\abs{x}},
          \frac{\Phi_2(x,t\abs{x}^\beta)}{\abs{x}^\beta}\right) = f(t),
\text{ where }
f = 
\begin{cases}
f_+, & \text{ if } \,x > 0,\\
f_-, & \text{ if } \,x < 0.
\end{cases}
\end{equation}
Hence, for each \(t\in\R\) and \(x\neq 0\) sufficiently small, 
\(y = \Phi_2(x,t \abs{x}^\beta)/\abs{x}^\beta\)
is a zero of the nonconstant polynomial 
\(H_{t,x}(y) \coloneqq G(\tilde x, y) - f(t) \in \R[y]\),
where \(\tilde x \coloneqq \Phi_1(x,t \abs{x}^\beta)/\abs{x}\).
Applying Cauchy's bound on the roots of a polynomial, we obtain
\begin{equation}
\label{eq: Cauchy's bound for Phi_2}
\abs{\frac{\Phi_2(x,t \abs{x}^\beta)}{\abs{x}^\beta}} \leq
1 + \max\left\{ 
     \abs{\frac{c_{m-1}}{c_m}}\abs{\tilde x}^r, \ldots, 
     \abs{\frac{c_{1}}{c_m}}\abs{\tilde x}^{r(m-1)},
     \abs{\frac{c_0{\tilde x}^d - f(t)}{c_m{\tilde x}^{d - rm}}}
     \right\} \,.
\end{equation}
Recall that we already know that \(\lim_{x\to 0^+}\Phi_1(x,0)/x\neq 0\) \,and \,\(\lim_{x\to 0^-}\Phi_1(x,0)/x\neq 0\). Then, by (\ref{eq: equality of first components of initial velocities}), we obtain
\begin{equation}
\label{eq: nonzero x tilde}
\lim_{x\to 0^+}\frac{\Phi_1(x,t\abs{x}^\beta)}{\abs{x}}\neq 0
\quad\text{and}\quad
\lim_{x\to 0^-}\frac{\Phi_1(x,t\abs{x}^\beta)}{\abs{x}}\neq 0,
\quad\text{for all }t\in\R.
\end{equation}
From (\ref{eq: Cauchy's bound for Phi_2}) and (\ref{eq: nonzero x tilde}), it follows that, for each \(t\in\R\), \(\Phi_2(x,t \abs{x}^\beta)/\abs{x}^\beta\) is bounded for \(x\neq 0\) sufficiently small.

Finally, letting \(x\to 0^+\) in (\ref{eq: gophi = lambda f prior to the limit}), we obtain
\begin{equation}
\label{eq: g phi+}
g(\phi_+(t)) = \abs{\lambda_+}^{-d}f_+(t),
\text{ where }
g = 
\begin{cases}
g_+, & \text{ if }\lambda_+ > 0,\\
g_-, & \text{ if }\lambda_+ < 0.
\end{cases}
\end{equation}
Similarly, letting \(x\to 0^-\) in (\ref{eq: gophi = lambda f prior to the limit}), we obtain
\begin{equation}
\label{eq: g phi-}
g(\phi_-(t)) = \abs{\lambda_-}^{-d}f_-(t),
\text{ where }
g = 
\begin{cases}
g_-, & \text{ if }\lambda_- > 0,\\
g_+, & \text{ if }\lambda_- < 0.
\end{cases}
\end{equation}
Clearly, equations (\ref{eq: g phi+}) and (\ref{eq: g phi-}) yield (\ref{eq: gophi = lambda f}).
\end{proof}

\begin{theorem}
\label{thm: equivalent polynomials, Lipschitz equivalent height functions}
Let \(F(X,Y)\) and \(G(X,Y)\) be \(\beta\)-quasihomogeneous polynomials of degree \(d\) that are not of the form \(cX^d\). Suppose that \(F\) and \(G\) are \({\cal R}\)-semialgebraically Lipschitz equivalent, and let  
\(\Phi\from(\R^2,0)\to(\R^2,0)\) be a germ of semialgebraic bi-Lipschitz homeomorphism such that \(G\circ\Phi = F\). If any of the following conditions is satisfied, then \(\Phi\) and \(\Phi^{-1}\) are \(\beta\)-regular. 
\begin{enumerate}[label=\textnormal{(\alph*)}]
\item Each of the height functions of \(F\) has at least one real zero and \(X\) is not a factor of \(G\).
\item Each of the height functions of \(F\) has at least two distinct real zeros.
\end{enumerate}
In the affirmative case, the following holds:
\begin{equation*}
\begin{cases}
f_+\cong g_+ \text{ and } \,f_-\cong g_-, & \text{if\/ \(\Phi\) is direct},\\ 
f_+\cong g_- \text{ and } \,f_-\cong g_+, & \text{if\/ \(\Phi\) is reverse}.
\end{cases}
\end{equation*}
Here, \(f_+,f_-\) are the height functions of\/ \(F\), and \(g_+,g_-\) are the height functions of\/ \(G\).
\end{theorem}

\begin{proof}
Suppose that either (a) or (b) holds. Then it follows that there exist a regular component \(R_t\) and a regular component \(L_t\) of \(F^{-1}(0)\setminus\{(0,0)\}\) that are matched by \(\Phi\) to regular components of \({G^{-1}(0)\setminus\{(0,0)\}}\). Thus, by Lemma \ref{lemma: characterization of beta-isomorphism via branch matching}, \(\Phi\) is \(\beta\)-regular, and we have:
\begin{equation*}
\begin{cases}
g_+\circ \phi_+ = \abs{\lambda_+}^{-d} f_+ 
\quad\text{and\/}\quad
g_-\circ \phi_- = \abs{\lambda_-}^{-d} f_-,
&\text{if\/ \(\Phi\) is direct},\\[2pt]
g_-\circ \phi_+ = \abs{\lambda_+}^{-d} f_+ 
\quad\text{and\/}\quad
g_+\circ \phi_- = \abs{\lambda_-}^{-d} f_-,
&\text{if\/ \(\Phi\) is reverse},
\end{cases}
\end{equation*}
where, \(f_+,f_-\) are the height functions of\/ \(F\), \(g_+,g_-\) are the height functions of\/ \(G\), and \((\lambda,\phi)\) is the \(\beta\)-transform of \(\Phi\).
Furthermore, since \(F\) is not of the form \(cX^d\), \(F\circ\Phi^{-1} = G\), and \(\Phi^{-1}\) matches regular components of \({G^{-1}(0)\setminus\{(0,0)\}}\), both in the right and left half-planes, to regular components of \(F^{-1}(0)\setminus\{(0,0)\}\), we can apply Lemma \ref{lemma: characterization of beta-isomorphism via branch matching} to conclude that \(\Phi^{-1}\) is also \(\beta\)-regular. Then, by Proposition \ref{prop: about lambda and phi in a beta-transform}, \(\lambda_+\) and \(\lambda_-\) are nonzero real numbers, and \(\phi_+\) and \(\phi_-\) are bi-Lipschitz. This completes the proof.
\end{proof}

\subsection{The zygothetic group}
\label{subsection: proto-transitions}

In this subsection, we express the condition that the height functions of two \(\beta\)-quasihomogeneous polynomials \(F,G\in\R[X,Y]\) of degree \(d\) can be arranged into pairs of Lipschitz equivalent functions in terms of a group action on \(\R^\R\times\R^\R\) (see Corollary \ref{cor: condition on height functions in terms of group action}).\\ 

Let \(\R^*\) be the multiplicative group of all nonzero real numbers, and let \({\cal L}\) be the group of all bi-Lipschitz homeomorphisms from \(\R\) to \(\R\). Let \(H\coloneqq \{(\lambda_1,\lambda_2) \in \R^*\times\R^*:
\lambda_1\lambda_2 > 0\)\} (considered as a subgroup of the direct product 
\(\R^*\times\R^*\)), and let \(K\coloneqq {\cal L} \times {\cal L}\) (direct product). 

\begin{definition}
The elements of \(H\times K\) are called zygotheties.\footnote{The term {\it zygothety} is derived from the Greek words \textgreek{ζυγός} (zygos), meaning {\it yoke}—a device that joins two things together, especially oxen in plowing—and \textgreek{θέσις} (thesis), meaning {\it placement} or {\it arrangement}. This name reflects the essence of the group operation, which involves a coupling along with a specific transformation.} Define a binary operation on \(H\times K\) by setting:
\begin{equation*}
(\mu,\psi)\circ(\lambda,\phi)\coloneqq
\begin{cases}
\left( (\lambda_1\mu_1, \lambda_2\mu_2), 
    (\psi_1\circ \phi_1,\psi_2\circ\phi_2) \right),
    & \text{if \(\lambda > 0\)}\\
\left( (\lambda_1\mu_2, \lambda_2\mu_1), 
    (\psi_2\circ \phi_1,\psi_1\circ\phi_2) \right),
    & \text{if \(\lambda < 0\)}
\end{cases}
\end{equation*}
for all \((\lambda,\phi) = ((\lambda_1,\lambda_2),(\phi_1,\phi_2))\) and 
\((\mu,\psi) = ((\mu_1,\mu_2),(\psi_1,\psi_2))\),
where \(\lambda > 0\) means that \(\lambda_1 > 0\) and \(\lambda_2 > 0\),
and \(\lambda < 0\) means that \(\lambda_1 < 0\) and \(\lambda_2 < 0\).
\end{definition}

\begin{proposition}
\label{prop: zygothetic group}
\((H\times K, \circ)\) is a group, called the zygothetic group. 
\end{proposition}
\begin{proof}
Let \((H\times K, \cdot)\) be the direct product of \(H\) and \(K\), so that
\begin{equation*}
(\mu,\psi)\cdot(\lambda,\phi) = 
\left( (\lambda_1\mu_1,\lambda_2\mu_2), (\psi_1\circ\phi_1,\psi_2\circ\phi_2) \right).
\end{equation*}

We express the operation \(\circ\) in terms of the operation \(\cdot\)\,, and then we use this expression to show that \((H\times K,\circ)\) is a group. 
Let \(\iota\from H\times K\to H\times K\) be the identity map and 
\(\tau\from H\times K\to H\times K\) be given by 
\(\tau((\lambda_1,\lambda_2),(\phi_1,\phi_2)) = 
((\lambda_2,\lambda_1),(\phi_2,\phi_1))\).

Define \(\theta\from H\to \text{Aut}(H\times K)\) by
\begin{equation*}
\theta(\lambda) \coloneqq
\begin{cases}
\iota, &\text{if }\lambda > 0,\\
\tau, &\text{if }\lambda < 0. 
\end{cases}
\end{equation*}
Clearly, \(\theta\) is a group homomorphism.

Let \(\pi\from H\times K\to H\) be the natural projection onto \(H\). Then
\(\alpha \coloneqq \theta\circ\pi\from H\times K\to \text{Aut}(H\times K)\)
is a group homomorphism such that:
\begin{enumerate}[label=\roman*.]
\item \(\alpha\circ\varphi = \alpha\) for all \(\varphi\in \text{Im}\,\alpha\);
\item \(\text{Im}\,\alpha\) is an abelian subgroup of \(\text{Aut}(H\times K)\).
\end{enumerate}
Also, we have:
\begin{equation*}
(\mu,\psi)\circ(\lambda,\phi) = 
   \left(\alpha(\lambda,\phi)(\mu,\psi)\right)\cdot(\lambda,\phi).
\end{equation*}
Building on these observations, the verification of the group axioms for \((H\times K, \circ)\) follows naturally. Note that the identity element of \((H\times K,\circ)\) coincides with the identity element of \((H\times K,\cdot)\), and that the inverse of an element \(g\) in \((H\times K,\circ)\) is given by \(\alpha(g^{-1})(g^{-1})\).
\end{proof}

Now, we define a family of actions of the zygothetic group on the set \(\R^\R\times\R^\R\).

\begin{proposition}
For each integer \(d\geq 1\), the map \(\circ\from (\R^\R\times\R^\R)\times (H\times K) \to \R^\R\times\R^\R\) defined by
\begin{equation*}
(g_1,g_2)\circ (\lambda,\phi) \coloneqq
\begin{cases}
\left(\abs{\lambda_1}^d g_1\circ\phi_1, \abs{\lambda_2}^d g_2\circ\phi_2\right),
&\text{if}\;\;\lambda > 0\\[5pt]
\left(\abs{\lambda_1}^d g_2\circ\phi_1, \abs{\lambda_2}^d g_1\circ\phi_2\right),
&\text{if}\;\;\lambda < 0
\end{cases}
\end{equation*}
is an action of the zygothetic group on \(\R^\R\times\R^\R\).
\end{proposition}
\begin{proof}
First, we note that the map
\(\cbullet\from (\R^\R\times\R^\R)\times (H\times K) \to \R^\R\times\R^\R\) given by
\begin{equation*}
(g_1,g_2)\cbullet (\lambda,\phi)\coloneqq
(\abs{\lambda_1}^d g_1\circ\phi_1, \abs{\lambda_2}^d g_2\circ\phi_2)
\end{equation*}
is an action of the direct product \((H\times K,\cdot)\) on 
\(\R^\R\times\R^\R\). 

Next, we express the map \(\circ\) in terms of the action \(\cbullet\), and then we use this expression to show that the map \(\circ\) is an action of the zygothetic group on \(\R^\R\times\R^\R\). Denote by \(\text{Bij}(\R^\R\times\R^\R)\) the group of all bijections on \(\R^\R\times\R^\R\). Let \(I\from \R^\R\times\R^\R\to \R^\R\times\R^\R\) be the identity map and \(T\from \R^\R\times\R^\R\to \R^\R\times\R^\R\) be given by 
\(T(g_1,g_2) = (g_2,g_1)\). Define 
\(\Theta\from H\to\text{Bij}(\R^\R\times\R^\R)\) by
\begin{equation*}
\Theta(\lambda) \coloneqq
\begin{cases}
I, &\text{if}\quad\lambda > 0,\\
T, &\text{if}\quad\lambda < 0.\\
\end{cases}
\end{equation*}
Clearly, \(\Theta\) is a group homomorphism. Let \(\pi\from H\times K\to H\) be the natural projection. Then 
\(A\coloneqq \Theta\circ\pi
     \from (H\times K,\circ) \to \text{Bij}(\R^\R\times\R^\R)\)
is a group homomorphism such that:
\vspace{0.2cm}
\begin{enumerate}[label=\Roman*.,itemsep=0.2cm]
\item \(A(\mu,\psi)((g_1,g_2)\cbullet(\lambda,\phi)) = 
(A(\mu,\psi)(g_1,g_2))\cbullet\alpha(\mu,\psi)(\lambda,\phi)\),
where \(\alpha = \theta\circ\pi\from H\times K\to \text{Aut}(H\times K,\cdot)\) is the homomorphism defined in the proof of Proposition \ref{prop: zygothetic group};
\item \(\text{Im}\,A\) is an abelian subgroup of 
\(\text{Bij}(\R^\R\times\R^\R)\).
\end{enumerate}
\vspace{0.2cm}
Also, we have:
\begin{equation*}
(g_1,g_2)\circ(\lambda,\phi) = 
     (A(\lambda,\phi)(g_1,g_2))\cbullet(\lambda,\phi).
\end{equation*}
Building on these observations, the verification of the group action axioms for the map \(\circ\) follows naturally.
\end{proof}

\begin{proposition}
Let \(f_1,f_2,g_1,g_2\from\R\to\R\) be arbitrary functions, and let 
\(d\geq 1\) be a fixed integer. The following conditions are equivalent:
\begin{enumerate}[label=\roman*.]
\item There exist constants \(c_1,c_2 > 0\) and bi-Lipschitz homeomorphisms 
\(\phi_1,\phi_2\from\R\to\R\) such that
\begin{equation*}
(g_1\circ \phi_1 = c_1 f_1 \;\text{ and }\; g_2\circ \phi_2 = c_2 f_2)
 \;\text{ or }\;
(g_2\circ \phi_1 = c_1 f_1 \;\text{ and }\; g_1\circ \phi_2 = c_2 f_2).
\end{equation*}
\item There exists a zygothety \((\lambda,\phi)\) such that
\((g_1,g_2)\circ (\lambda,\phi) = (f_1,f_2)\).
\end{enumerate}
\end{proposition}

\begin{proof}
We prove only that (i) implies (ii), the other implication being immediate. Suppose that there exist constants \(c_1,c_2 > 0\) and bi-Lipschitz homeomorphisms
\(\phi_1,\phi_2\from\R\to\R\) such that 
\begin{equation*}
(g_1\circ \phi_1 = c_1 f_1 \;\text{ and }\; g_2\circ \phi_2 = c_2 f_2)
 \;\text{ or }\;
(g_2\circ \phi_1 = c_1 f_1 \;\text{ and }\; g_1\circ \phi_2 = c_2 f_2).
\end{equation*}
If \(g_1\circ\phi_1 = c_1 f_1\) and \(g_2\circ\phi_2 = c_2 f_2\), then
\((g_1,g_2)\circ(\lambda,\phi) = (f_1,f_2)\),
where \(\lambda = (c_1^{-1/d},c_2^{-1/d})\) and \(\phi = (\phi_1,\phi_2)\).
Otherwise, if \(g_2\circ\phi_1 = c_1 f_1\) and \(g_1\circ\phi_2 = c_2 f_2\), 
then \((g_1,g_2)\circ(\lambda,\phi) = (f_1,f_2)\),
where \(\lambda = (-c_1^{-1/d},-c_2^{-1/d})\) and \(\phi = (\phi_1,\phi_2)\).
\end{proof}

\begin{corollary}
\label{cor: condition on height functions in terms of group action}
Let \(F(X,Y)\) and \(G(X,Y)\) be \(\beta\)-quasihomogeneous polynomials of degree \(d\). Let \(f_+,f_-\) be the height functions of \(F\), and 
\(g_+,g_-\) be the height functions of \(G\). The following conditions are equivalent:
\begin{enumerate}[label=\roman*.]
\item \((f_+\cong g_+ \,\text{ and }\, f_-\cong g_-)\) 
\,or\, 
\((f_+\cong g_- \,\text{ and }\, f_-\cong g_+)\).

\item \((g_+,g_-)\circ (\lambda,\phi) = (f_+,f_-)\), for some zygothety 
\((\lambda,\phi)\).
\end{enumerate}
\end{corollary}

\subsection{\(\beta\)-regular zygotheties and the inverse \(\beta\)-transform}
\label{subsection: beta-transitions and the inverse beta-transform}

In this subsection, we show that if the height functions of two \(\beta\)-quasihomogeneous polynomials \(F,G\in\R[X,Y]\) of degree \(d\), neither of which is of the form \(cX^d\), are related via a zygothety \((\lambda,\phi)\) satisfying a specific condition, called  \(\beta\)-regularity, then \(F\) and \(G\) are \({\cal R}\)-semialgebraically Lipschitz equivalent (see Theorem \ref{thm: beta-regular zygothety, R-semialgebraic Lipschitz equivalence}).

\begin{definition}
A zygothety \((\lambda,\phi)\) is said to be {\it \(\beta\)-regular} if the following conditions are satisfied:
\begin{enumerate}[label = \roman*.]
\item The limits\/ 
\(\lim_{\abs{t}\to+\infty} \phi_1(t)/t\) 
\,and 
\,\(\lim_{\abs{t}\to+\infty} \phi_2(t)/t\)
exist and are nonzero.
\item \(\displaystyle 
   \abs{\lambda_1}^\beta\cdot\lim_{\abs{t}\to+\infty} \frac{\phi_1(t)}{t} = 
\abs{\lambda_2}^\beta\cdot\lim_{\abs{t}\to+\infty} \frac{\phi_2(t)}{t}\)
\end{enumerate}
\end{definition}

\begin{proposition}
The \(\beta\)-regular zygotheties form a subgroup of the zygothetic group.
\end{proposition}

\begin{proof}
It is immediate that the \(\beta\)-regularity conditions are satisfied for \(\lambda_1=\lambda_2=1\) and \(\phi_1=\phi_2=\id_\R\). Hence, the identity element \(\hat\imath \coloneqq ((1,1),(\id_\R,\id_\R))\) of the zygothetic group is \(\beta\)-regular.

Now, we show that if \((\lambda,\phi)\) and \((\mu,\psi)\) are \(\beta\)-regular zygotheties, then \((\mu,\psi)\circ (\lambda,\phi)\) is a \(\beta\)-regular zygothety. Let \((\lambda,\phi) = ((\lambda_1,\lambda_2),(\phi_1,\phi_2))\) 
\,and 
\,\((\mu,\psi) = ((\mu_1, \mu_2),(\psi_1,\psi_2))\) 
be \(\beta\)-regular zygotheties.
By definition,
\begin{equation*}
(\mu,\psi)\circ (\lambda,\phi) = 
\begin{cases}
((\lambda_1\mu_1,\lambda_2\mu_2),(\psi_1\circ\phi_1,\psi_2\circ\phi_2)), &\text{if }\lambda > 0,\\
((\lambda_1\mu_2,\lambda_2\mu_1),(\psi_2\circ\phi_1,\psi_1\circ\phi_2)), &\text{if }\lambda < 0.\\
\end{cases}
\end{equation*}

Suppose that \(\lambda > 0\). For \(i = 1,2\), we have:
\begin{equation*}
\lim_{\abs{t}\to+\infty}\frac{\psi_i(\phi_i(t))}{t} = 
    \lim_{\abs{t}\to+\infty} \frac{\psi_i(\phi_i(t))}{\phi_i(t)} 
        \cdot \lim_{\abs{t}\to+\infty} \frac{\phi_i(t)}{t} = 
    \lim_{\abs{t}\to+\infty} \frac{\psi_i(t)}{t} 
        \cdot \lim_{\abs{t}\to+\infty} \frac{\phi_i(t)}{t}
    \neq 0.
\end{equation*}
Also, we have:  
\begin{multline*}
\abs{\lambda_1\mu_1}^\beta \lim_{\abs{t}\to+\infty} \frac{\psi_1(\phi_1(t))}{t}
   = \abs{\lambda_1\mu_1}^\beta \lim_{\abs{t}\to+\infty} \frac{\psi_1(t)}{t} 
        \cdot \lim_{\abs{t}\to+\infty} \frac{\phi_1(t)}{t}
   = \left(\abs{\mu_1}^\beta\frac{\psi_1(t)}{t} \right)\cdot
         \left(\abs{\lambda_1}^\beta\frac{\phi_1(t)}{t} \right)\\[\lineskip]
   = \left(\abs{\mu_2}^\beta\frac{\psi_2(t)}{t} \right)\cdot
         \left(\abs{\lambda_2}^\beta\frac{\phi_2(t)}{t} \right)
   = \abs{\lambda_2\mu_2}^\beta \lim_{\abs{t}\to+\infty} \frac{\psi_2(t)}{t} 
        \cdot \lim_{\abs{t}\to+\infty} \frac{\phi_2(t)}{t}
   =\abs{\lambda_2\mu_2}^\beta 
           \lim_{\abs{t}\to+\infty} \frac{\psi_2(\phi_2(t))}{t}.
\end{multline*}
Hence, \((\mu,\psi)\circ (\lambda,\phi)\) \(\beta\)-regular. The proof for \(\lambda < 0\) is analogous.

Finally, we show that if \((\lambda,\phi)\) is a \(\beta\)-regular zygothety, then \((\lambda,\phi)^{-1}\) is also \(\beta\)-regular. Let \((\lambda,\phi) = ((\lambda_1,\lambda_2),(\phi_1,\phi_2))\) be a \(\beta\)-regular zygothety. 
The inverse zygothety of \((\lambda,\phi)\) is given by:\footnote{See the end of the proof of Proposition \ref{prop: zygothetic group}}
\begin{equation*}
(\lambda,\phi)^{-1} = 
\begin{cases}
((\lambda_1^{-1},\lambda_2^{-1}),(\phi_1^{-1},\phi_2^{-1}))
     &\text{if }\lambda > 0,\\
((\lambda_2^{-1},\lambda_1^{-1}),(\phi_2^{-1},\phi_1^{-1}))
     &\text{if }\lambda < 0.\\
\end{cases}
\end{equation*}

Suppose that \(\lambda > 0\). For \(i=1,2\), we have:
\begin{equation*}
\lim_{\abs{t}\to+\infty} \frac{\phi_i(t)}{t}  
   =\lim_{\abs{t}\to+\infty} \frac{\phi_i(\phi_i^{-1}(t))}{\phi_i^{-1}(t)}
   =\lim_{\abs{t}\to+\infty} \frac{t}{\phi_i^{-1}(t)}\,,
\end{equation*}
because \(\lim_{\abs{t}\to+\infty} \abs{\phi_i^{-1}(t)} = +\infty\).
Hence, 
\begin{equation*}
\lim_{\abs{t}\to+\infty} \frac{\phi_i^{-1}(t)}{t} = 
\left(\lim_{\abs{t}\to+\infty} \frac{\phi_i(t)}{t}\right)^{-1} \neq 0,
\quad\text{for }i=1,2.
\end{equation*}
Also,  
\begin{multline*}
\abs{\lambda_1^{-1}}^\beta \lim_{\abs{t}\to+\infty} \frac{\phi_1^{-1}(t)}{t}
   = \abs{\lambda_1}^{-\beta} 
       \left(\lim_{\abs{t}\to+\infty} \frac{\phi_1(t)}{t}\right)^{-1}
   = \left(\abs{\lambda_1}^\beta\lim_{\abs{t}\to+\infty} \frac{\phi_1(t)}{t}\right)^{-1}\\
   = \left(\abs{\lambda_2}^\beta\lim_{\abs{t}\to+\infty} \frac{\phi_2(t)}{t}\right)^{-1}
   = \abs{\lambda_2}^{-\beta} 
       \left(\lim_{\abs{t}\to+\infty} \frac{\phi_2(t)}{t}\right)^{-1}
   = \abs{\lambda_2^{-1}}^\beta 
       \lim_{\abs{t}\to+\infty} \frac{\phi_2^{-1}(t)}{t}.   
\end{multline*}
Hence, \((\lambda,\phi)^{-1}\) is \(\beta\)-regular. 
The proof for \(\lambda < 0\) is analogous.
\end{proof}

\begin{definition}
\label{def: inverse beta-transform}
Given a \(\beta\)-regular zygothety \((\lambda,\phi) = ((\lambda_1,\lambda_2),(\phi_1,\phi_2))\), we define a map 
\(\Phi\from\R^2\to\R^2\) by setting:
\begin{itemize}
\item \(\Phi(x,t\abs{x}^\beta) \coloneqq 
\left(\lambda_1 x, \abs{\lambda_1}^\beta \phi_1(t) \abs{x}^\beta\right),\;
\text{for } x > 0,\, t\in\R\);

\item \(\Phi(x,t\abs{x}^\beta) \coloneqq 
\left(\lambda_2 x, \abs{\lambda_2}^\beta \phi_2(t) \abs{x}^\beta\right),\;
\text{for } x < 0,\, t\in\R\);

\item \(\displaystyle\Phi(0,y) \coloneqq 
\left(0,\abs{\lambda_1}^\beta\lim_{\abs{t}\to+\infty} \frac{\phi_1(t)}{t}\,y\right)
= 
\left(0,\abs{\lambda_2}^\beta\lim_{\abs{t}\to+\infty} \frac{\phi_2(t)}{t}\,y\right),
\;\text{for all }y\in\R\).
\end{itemize}
The germ 
\(\Phi\from(\R^2,0)\to(\R^2,0)\) determined by the map \(\Phi\) 
is called the {\it inverse \(\beta\)-transform} of \((\lambda,\phi)\).
\end{definition}

\begin{proposition}
The inverse \(\beta\)-transform of the identity zygothety \(\hat\imath\) is the germ of the identity map \(I\from (\R^2,0)\to(\R^2,0)\).
\end{proposition}
\begin{proof}
The result follows by directly applying Definition \ref{def: inverse beta-transform} and verifying that each case of \(\Phi\) matches the identity map when \((\lambda, \phi)\) is the identity \(\beta\)-transition. 
\end{proof}

\begin{proposition}
Let \((\lambda,\phi)\) and \((\mu,\psi)\) be \(\beta\)-regular zygotheties,
let \(\Phi\) and \(\Psi\) be their respective inverse \(\beta\)-transforms, and let 
\(Z\) be the inverse \(\beta\)-transform of \((\mu,\psi)\circ(\lambda,\phi)\).
Then \(Z = \Psi\circ\Phi\).
\end{proposition}

\begin{proof}
This follows directly from Definition \ref{def: inverse beta-transform} by verifying the composition of the maps case by case.
\end{proof}

\begin{corollary}
\label{cor: Inverse beta-transform of the inverse of a beta-transition}
Let \((\lambda, \phi)\) be a \(\beta\)-transition with inverse \(\beta\)-transform \(\Phi\).  Then the inverse \(\beta\)-transform of the inverse zygothety \((\lambda,\phi)^{-1}\) is \(\Phi^{-1}\).
\end{corollary}

The next two results provide the final ingredients for the proof of Theorem \ref{thm: beta-regular zygothety, R-semialgebraic Lipschitz equivalence}, the main result of this subsection.

\begin{lemma}
\label{lemma: boundedness of phi(t) - t phi'(t)}
Let \(\phi\from\R\to\R\) be a bi-Lipschitz function such that \(g\circ \phi = f\) for some nonconstant polynomial functions \(f,g\from\R\to\R\). 
Then \(\phi\) is bi-analytic and \(\phi(t) - t\phi^\prime(t)\) is bounded.
\end{lemma}

\begin{proof}
First, note that by Lemma \ref{lemma: bi-Lipschitz invariance of degree},
\(\deg f = \deg g\), so we can apply Lemma \ref{lemma: bi-Lipschitz iff bi-analytic} to conclude that \(\phi\) is bi-analytic. 
Now, by Lemma \ref{lemma: analytic extension of phi(t)/t}, the function
\(\psi\from\R\to\R\) given by
\begin{equation*}
\psi(t) \coloneqq
\begin{cases}
t\cdot\phi(t^{-1}), &\text{if }t\in\R\setminus\{0\}\\
\lim_{\abs{t}\to+\infty} \phi(t)/t, &\text{if }t = 0\\
\end{cases}
\end{equation*}
is analytic.

From the definition of \(\psi\), it is immediate that
\begin{equation*}
\frac{\phi(t)}{t} = \psi(t^{-1})\quad\text{for all }t\in\R\setminus\{0\}.
\end{equation*}
Differentiating both sides of this identity, we get
\begin{equation*}
\frac{\phi^\prime(t)\cdot t - \phi(t)}{t^2} = -\frac{\psi^\prime(t^{-1})}{t^2}
\quad\text{for all }t\in\R\setminus\{0\}.
\end{equation*}
Equivalently, we have
\begin{equation*}
\phi(t) - t\phi^\prime(t) = \psi^\prime(t^{-1})
\quad\text{for all }t\in\R\setminus\{0\}.
\end{equation*}
Hence,
\begin{equation*}
\lim_{\abs{t}\to+\infty} \phi(t) - t\phi^\prime(t) = \psi^\prime(0).
\end{equation*}
Since the function \(\phi(t) - t\phi^{\prime}(t)\) is continuous on \(\R\), the existence of this limit implies that this function is bounded.
\end{proof}

\begin{lemma}
\label{lemma: Lipschitz on the strip H_delta, right half-plane}
Let \(\phi\from\R\to\R\) be a bi-Lipschitz function such that \(g\circ \phi = f\) for some nonconstant polynomial functions \(f,g\from\R\to\R\), and let \(\lambda\) be a nonzero real number. Then the map 
\(\Phi\from {\cal H}\to\R^2\), defined on the right half-plane 
\({\cal H}\coloneqq \{(x,y)\in\R^2: x > 0\}\) by
\begin{equation*}
\Phi(x,tx^\beta) \coloneqq (\lambda x, \abs{\lambda}^\beta\phi(t)x^\beta)
\end{equation*}
for all \(x>0\) and \(t\in\R\), is Lipschitz on the strip 
\({\cal H}_\delta \coloneqq \{(x,y)\in\R^2 : 0 < x < \delta\}\), for each 
\(\delta > 0\).
\end{lemma}

\begin{proof}
Let \(\delta > 0\) be fixed. We prove that \(\Phi\) is Lipschitz on both the upper half-strip \({\cal H}_\delta\cap \{y > 0\}\) and the 
lower half-strip \({\cal H}_\delta\cap \{y < 0\}\). Let us see that this implies the result. Assuming this claim, and using the fact that \(\Phi\) is continuous, we see that there exists a constant \(C > 0\) such that
\begin{equation}
\label{eq: Lipschitz condition for Phi on a quadrant}
\abs{\Phi(x_1,y_1) - \Phi(x_2,y_2)} \leq C\abs{(x_1,y_1) - (x_2,y_2)},
\end{equation} 
whenever \((x_1,y_1)\) and \((x_2,y_2)\) both belong to 
\({\cal H}_\delta\cap\{y\geq 0\}\) or to \({\cal H}_\delta\cap\{y\leq 0\}\).
We show that (\ref{eq: Lipschitz condition for Phi on a quadrant}) still holds for \((x_1,y_1)\in{\cal H}_\delta\cap\{y\geq 0\}\) 
and \((x_2,y_2)\in{\cal H}_\delta\cap\{y\leq 0\}\). 
Let \((\bar{x},0)\) be the point at which the line segment whose endpoints are \((x_1,y_1)\) and \((x_2,y_2)\) intersects the \(x\)-axis. By our assumptions, we have:
\begin{equation*}
\abs{\Phi(x_1,y_1) - \Phi(\bar{x},0)} \leq C\abs{(x_1,y_1) - (\bar{x},0)}
\end{equation*} 
and
\begin{equation*}
\abs{\Phi(\bar{x},0) - \Phi(x_2,y_2) } \leq C\abs{(\bar{x},0) - (x_2,y_2)}.
\end{equation*} 
Hence,
\begin{align*}
\abs{\Phi(x_1,y_1) - \Phi(x_2,y_2)} &\leq 
\abs{\Phi(x_1,y_1) - \Phi(\bar{x},0)} + \abs{\Phi(\bar{x},0) - \Phi(x_2,y_2)}\\
&\leq C\left(\abs{(x_1,y_1) - (\bar{x},0)} + \abs{(\bar{x},0) - (x_2,y_2)}\right)\\
&=C\abs{(x_1,y_1) - (x_2,y_2)}, 
\end{align*}
where the last equality holds because the point \((\bar{x},0)\) lies in the segment whose endpoints are \((x_1,y_1)\) and \((x_2,y_2)\). Therefore, our initial claim implies that \(\Phi\) is Lipschitz on the strip 
\({\cal H}_\delta\).

In order to establish our initial claim, we first show that for each fixed pair of points \((x_1,t_1x_1^\beta)\) and \((x_2,t_2x_2^\beta)\), either both on 
\({\cal H}_\delta\cap\{y > 0\}\) or both on \({\cal H}_\delta\cap\{y < 0\}\),
with \(x_1\neq x_2\)\/ and\/ \(t_1\neq t_2\),
there exist \(\omega\) between \(x_1\) and \(x_2\), and \(\tau\) between \(t_1\) and \(t_2\) such that\footnote{By Lemma \ref{lemma: boundedness of phi(t) - t phi'(t)}, \(\phi\) is bi-analytic. So, in particular, \(\phi^\prime\from\R\to\R\) is a well-defined continuous function.} 
\begin{equation}
\label{eq: algebraic trick}
\phi(t_2)x_2^\beta - \phi(t_1)x_1^\beta = 
(\phi(\tau) - \tau \phi^\prime(\tau))\cdot\beta\omega^{\beta - 1}\cdot(x_2 - x_1) + \frac{\phi(t_2) - \phi(t_1)}{t_2 - t_1}\cdot(t_2x_2^\beta - t_1x_1^\beta)
\end{equation}
In fact, 
\begin{align}
\phi(t_2) x_2^\beta - \phi(t_1) x_1^\beta &= 
(\phi(t_1) + q\cdot (t_2 - t_1)) x_2^\beta - \phi(t_1) x_1^\beta,
\text{ where } q = \frac{\phi(t_2) - \phi(t_1)}{t_2 - t_1}\nonumber\\
&= \phi(t_1)\cdot(x_2^\beta - x_1^\beta) + q\cdot(t_2x_2^\beta - t_1x_2^\beta)\nonumber\\
&= \phi(t_1)\cdot(x_2^\beta - x_1^\beta) + q\cdot(t_2x_2^\beta - t_1x_1^\beta) + q\cdot t_1\cdot(x_1^\beta - x_2^\beta)\nonumber\\
&= (\phi(t_1) - q\cdot t_1)\cdot(x_2^\beta - x_1^\beta) + q\cdot(t_2x_2^\beta - t_1x_1^\beta)\nonumber\\
&= \frac{t_2\phi(t_1) - t_1\phi(t_2)}{t_2 - t_1}\cdot(x_2^\beta - x_1^\beta) + \frac{\phi(t_2) - \phi(t_1)}{t_2 - t_1}\cdot(t_2x_2^\beta - t_1x_1^\beta)
\label{eq: algebraic identity for the second coordinate of the difference}
\end{align}

Since the points \((x_1,t_1x_1^\beta)\) and \((x_2,t_2x_2^\beta)\) are either both on \({\cal H}_\delta\cap\{y > 0\}\) or both on \({\cal H}_\delta\cap\{y < 0\}\), the real numbers \(t_1,t_2\) are either both positive or both negative. Thus, by Pompeiu's Mean Value Theorem \cite[p.~1~--~2]{Drago}, there exists a real number \(\tau\) between \(t_1\) and \(t_2\) such that 
\begin{equation}
\label{eq: Pompeiu's Mean Value Theorem}
\frac{t_2\phi(t_1)-t_1\phi(t_2)}{t_2 - t_1} = 
\phi(\tau) - \tau\phi^\prime(\tau).
\end{equation}

Also, by Lagrange's Mean Value Theorem, there exists a real number 
\(\omega\) between \(x_1\) and \(x_2\) such that
\begin{equation}
\label{eq: Lagrange's Mean Value Theorem}
x_2^\beta - x_1^\beta =\beta\omega^{\beta-1}\cdot (x_2 - x_1).
\end{equation}
Substituting (\ref{eq: Pompeiu's Mean Value Theorem}) and (\ref{eq: Lagrange's Mean Value Theorem}) in (\ref{eq: algebraic identity for the second coordinate of the difference}), we obtain (\ref{eq: algebraic trick}).

Now, since \(\phi\) is Lipschitz, there exists a constant \(C_1 > 0\) such that
\begin{equation*}
\abs{\frac{\phi(t_2) - \phi(t_1)}{t_2 - t_1}} \leq C_1,
\end{equation*}
for \(t_1\neq t_2\). On the other hand, by Lemma \ref{lemma: boundedness of phi(t) - t phi'(t)}, \(\phi(t)-t\phi^\prime(t)\) is bounded, so there exists a constant \(C_2 > 0\) such that
\begin{equation*}
\abs{(\tau\phi^\prime(\tau) - \phi(\tau))\cdot\beta\omega^{\beta - 1}}
\leq C_2,
\end{equation*}
provided that \(0 < x_1,x_2 < \delta\).

Applying these bounds to (\ref{eq: algebraic trick}), we obtain:
\begin{equation*}
\abs{\phi(t_2) x_2^\beta - \phi(t_1) x_1^\beta} \leq 
C\cdot\left(\abs{x_2 - x_1} + \abs{t_2x_2^\beta - t_1x_1^\beta}\right),
\end{equation*}
where \(C = \max\{C_1,C_2\}\), for any pair of points \((x_1,t_1x_1^\beta)\) and \((x_2,t_2x_2^\beta)\), either both on \({\cal H}_\delta\cap{\{y > 0\}}\) or both on \({\cal H}_\delta\cap\{y < 0\}\). Our initial claim follows immediately from this inequality.
\end{proof}

\begin{theorem}
\label{thm: beta-regular zygothety, R-semialgebraic Lipschitz equivalence}
Let \(F, G\in\R[X,Y]\) be \(\beta\)-quasihomogeneous polynomials of degree \(d\) that are not of the form \(cX^d\), and let \(f_+,f_-\) be the height functions of\/ \(F\) and \(g_+,g_-\) be the height functions of\/ \(G\). If\/ \((g_+,g_-)\circ(\lambda,\phi) = (f_+,f_-)\)
for some \(\beta\)-regular zygothety \((\lambda,\phi)\), then the inverse \(\beta\)-transform \(\Phi\from(\R^2,0)\to(\R^2,0)\) is a germ of semialgebraic bi-Lipschitz map and \(G\circ\Phi = F\).
\end{theorem}

\begin{proof}
By definition, the inverse \(\beta\)-transform of the \(\beta\)-regular zygothety \((\lambda,\phi)\) is the germ \(\Phi\from(\R^2,0)\to(\R^2,0)\) of the map
\(\Phi\from\R^2\to\R^2\) defined by:
\begin{itemize}
\item \(\Phi(x,t\abs{x}^\beta) \coloneqq 
\left(\lambda_1 x, \abs{\lambda_1}^\beta \phi_1(t) \abs{x}^\beta\right),\;
\text{for } x > 0,\, t\in\R\);

\item \(\Phi(x,t\abs{x}^\beta) \coloneqq 
\left(\lambda_2 x, \abs{\lambda_2}^\beta \phi_2(t) \abs{x}^\beta\right),\;
\text{for } x < 0,\, t\in\R\);

\item \(\displaystyle\Phi(0,y) \coloneqq 
\left(0,\abs{\lambda_1}^\beta\lim_{\abs{t}\to+\infty} \frac{\phi_1(t)}{t}\,y\right)
= 
\left(0,\abs{\lambda_2}^\beta\lim_{\abs{t}\to+\infty} \frac{\phi_2(t)}{t}\,y\right),
\;\text{for all }y\in\R\).
\end{itemize}
Since \(\phi_1\) and \(\phi_2\) are both semialgebraic functions, it is immediate that \(\Phi\) is a semialgebraic map. Now, we must show that \(\Phi\) is a bi-Lipschitz homeomorphism from a neighborhood of the origin in \(\R^2\) to another neighborhood of the origin in \(\R^2\). Let us start by showing that \(\Phi\) is Lipschitz on the strip \(\{(x,y)\in\R^2 : \abs{x} < \delta\}\), for each \(\delta > 0\). Let \(\delta > 0\) be fixed arbitrarily.
By Lemma \ref{lemma: Lipschitz on the strip H_delta, right half-plane}, there exists a constant \(C_+ > 0\) such that 
\(\Phi\vert_{{\cal H}_{\delta}}\from{\cal H}_{\delta}\to\R^2\) is \(C_+\)-Lipschitz. Since \(\Phi\vert_{{\cal H}_{\delta}}\) is uniformly continuous and takes values in a complete metric space, it has a unique continuous extension 
\(\widetilde\Phi\) to 
\(\widetilde{{\cal H}}_{\delta}\coloneqq \{(x,y)\in\R^2 : 0\leq x < \delta\}\). Let us show that \(\widetilde\Phi = \Phi\vert_{\widetilde{{\cal H}}_{\delta}}\). Obviously, \(\widetilde{\Phi}(x,y) = \Phi(x,y)\) for all 
\((x,y)\in{\cal H}_{\delta}\). For all \(y\in\R\setminus\{0\}\), we have:
\begin{equation*}
\widetilde{\Phi}(0,y) = \lim_{x\to 0^+}\Phi(x,y)\\
= \lim_{x\to 0^+}\left(\lambda_1 x, 
\abs{\lambda_1}^\beta\cdot\frac{\phi_1(y/x^\beta)}{y/x^\beta}\cdot y\right)
= \left(0,
   \abs{\lambda_1}^\beta\cdot\lim_{\abs{t}\to+\infty}\frac{\phi_1(t)}{t}\cdot y\right)
= \Phi(0,y).
\end{equation*}
Finally, 
\begin{equation*}
\widetilde\Phi(0,0) = \lim_{x\to 0^+} \widetilde\Phi(x,0) = \lim_{x\to 0^+} \Phi(x,0) = \Phi(0,0).
\end{equation*}
Hence, \(\widetilde\Phi = \Phi\vert_{\widetilde{{\cal H}}_{\delta}}\). This means that \(\Phi\vert_{\widetilde{{\cal H}}_{\delta}}\) is continuous, and since 
\(\Phi\vert_{{\cal H}_{\delta}}\) is \(C_+\)-Lipschitz, it follows that 
\(\Phi\vert_{\widetilde{\cal H}_{\delta}}\) is \(C_+\)-Lipschitz too.
Similarly, we can prove that there exists \(C_- > 0\) such that 
\(\Phi\vert_{-\widetilde{\cal H}_{\delta}}\from 
-\widetilde{\cal H}_{\delta}\to\R^2\) is \(C_-\)-Lipschitz. Therefore, 
\(\Phi\) is \(C\)-Lipschitz on the strip \(\{(x,y)\in\R^2 : \abs{x} < \delta\}\), where \(C = \max\{C_+,C_-\}\).
Furthermore, since \((f_+,f_-)\circ(\lambda,\phi)^{-1} = (g_+,g_-)\) and 
\(\Phi^{-1}\) is the inverse \(\beta\)-transform of \((\lambda,\phi)^{-1}\), it follows from what we just proved that \(\Phi^{-1}\) is \(C^\prime\)-Lipschitz on the strip \(\{(x,y)\in\R^2 : \abs{x} < \delta\}\), for some constant \(C^\prime > 0\). This completes the proof that \(\Phi\from(\R^2,0)\to(\R^2,0)\) is a germ of semialgebraic bi-Lipschitz map.

Finally, we show that \(G\circ\Phi = F\). From the relation \((g_+,g_-)\circ(\lambda,\phi) = (f_+,f_-)\) and the fact that \(F\) and \(G\) are \(\beta\)-quasihomogeneous of degree \(d\), it follows naturally that
\begin{equation*}
G(\lambda_1 x, \abs{\lambda_1}^\beta\phi_1(t)\abs{x}^\beta) = 
F(x,t\abs{x}^\beta)\;\;\text{for }x > 0,t\in\R;
\end{equation*}
and
\begin{equation*}
G(\lambda_2 x,\abs{\lambda_2}^\beta\phi_2(t)\abs{x}^\beta) = 
F(x,t\abs{x}^\beta)\;\;\text{for }x < 0,t\in\R. 
\end{equation*}
Consequently,
\begin{equation*}
G(\Phi(x,y)) = F(x,y)
\;\;\text{for all }(x,y)\in\R^2, \text{ with }x\neq 0.
\end{equation*}
Since \(\Phi\) is continuous\footnote{Clearly, \(\Phi\) is continuous on the right half-plane \(\{(x,y)\in\R^2: x > 0\}\) and also on the left half-plane \(\{(x,y)\in\R^2: x < 0\}\). Also, \(\Phi\) is Lipschitz (and therefore continuous) on a strip \(\{(x,y)\in\R^2:\abs{x}<\delta\}\). Since the right half-plane, the left half-plane, and the strip around the \(y\)-axis form an open cover of the plane, it follows that \(\Phi\) is continuous.}, we have
\(G(\Phi(x,y)) = F(x,y)\) for all \((x,y)\in\R^2\). 
\end{proof}

\subsection{Some conditions for replacing zygotheties with \(\beta\)-regular zygotheties}
\label{subsection: shifting from proto to beta}

Let \(F,G\in\R[X,Y]\) be \(\beta\)-quasihomogeneous polynomials of degree \(d\). Suppose that the height functions of \(F\) and \(G\) can be arranged into pairs of Lipschitz equivalent functions.
Our goal is to find conditions under which this assumption implies that \(F\) and \(G\) are \({\cal R}\)-semialgebraically Lipschitz equivalent.
By Corollary \ref{cor: condition on height functions in terms of group action}, the height functions \(f_+,f_-\) of \(F\) and the height functions \(g_+,g_-\) of 
\(G\) can be arranged into pairs of Lipschitz equivalent functions if and only if 
\((g_+,g_-)\circ(\lambda,\phi) = (f_+,f_-)\) for some zygothety 
\((\lambda,\phi)\). In general, such a zygothety \((\lambda,\phi)\) is not necessarily \(\beta\)-regular, but since it is not uniquely determined by \((g_+,g_-)\) and \((f_+,f_-)\) we can still ask whether it may be replaced with a \(\beta\)-regular zygothety \((\tilde\lambda,\tilde\phi)\). In this subsection, we are interested in finding conditions under which the answer to this question is affirmative. Then we can apply Theorem \ref{thm: beta-regular zygothety, R-semialgebraic Lipschitz equivalence} to conclude that \(F\) and \(G\) are 
\({\cal R}\)-semialgebraically Lipschitz equivalent. 

We consider separately the case where \(F\) and \(G\) are both of the form 
\(cX^d\) and the case where neither \(F\) nor \(G\) is of this form.\footnote{If the height functions of \(F\) and \(G\) can be arranged into pairs of Lipschitz equivalent functions, then either both \(F\) and \(G\) are of the form \(cX^d\) or neither is of this form (see Corollary \ref{cor: beta-quasihomogeneous polynomials of the form cX^d}). The main reason for considering these two cases separately is that only in the second case are the height functions of \(F\) and \(G\) nonconstant.} In the first case, we can easily determine directly from first principles whether \(F\) and \(G\) are \({\cal R}\)-semialgebraically Lipschitz equivalent. In the second case, we follow the strategy outlined above: assuming that the height functions of \(F\) and \(G\) can be arranged into pairs of Lipschitz equivalent functions, so that \((g_+,g_-)\circ(\lambda,\phi) = (f_+,f_-)\), where \((\lambda,\phi)\) is a zygothety, we find conditions under which we can construct from \((\lambda,\phi)\) 
a \(\beta\)-regular zygothety \((\tilde\lambda,\tilde\phi)\) such that  
\((g_+,g_-)\circ(\tilde\lambda,\tilde\phi) = (f_+,f_-)\).
But first, we show that if the height functions of \(F\) and \(G\) can be arranged into pairs of Lipschitz equivalent functions, then either both 
\(F\) and \(G\) are of the form \(cX^d\), or neither is of this form.
This follows from the next two propositions (see Corollary \ref{cor: beta-quasihomogeneous polynomials of the form cX^d}).

\begin{proposition}
\label{prop: cX^d in terms of the multiplicity of X as a factor}
Let \(P\in\R[X,Y]\) be a \(\beta\)-quasihomogeneous polynomial of degree 
\(d\) and let \(e\) be the multiplicity of \(X\) as a factor of \(P\) in 
\(\R[X,Y]\). Then \(e\leq d\), with equality if and only if \(P\) is of the form 
\(cX^d\).
\end{proposition}

\begin{proof}
Since \(P\) is a \(\beta\)-quasihomogeneous polynomial of degree \(d\), we have \(P(X,Y) = \sum_{k = 0}^n c_kX^{d-rk}Y^{sk}\),
where \(c_n\neq 0\) and \(0\leq n\leq \lfloor d/r\rfloor\). 
Then, clearly, \(e = d - rn\). So \(e\leq d\), with equality if and only if \(n = 0\).
Since \(n = 0\) if and only if \(P\) is of the form \(cX^d\), the result follows.
\end{proof}

\begin{proposition}
\label{prop: invariance of the multiplicity of the factor X}
Let \(F,G\in\R[X,Y]\) be \(\beta\)-quasihomogeneous polynomials of degree 
\(d\). Let \(e_F\) and \(e_G\) denote the multiplicities of \(X\) as a factor of \(F\) and \(G\), respectively, so that
\(F(X,Y) = X^{e_F}\cdot\widetilde{F}(X,Y)\) and 
\(G(X,Y) = X^{e_G}\cdot\widetilde{G}(X,Y)\), where
\(X\nmid\widetilde{F}(X,Y)\) and \(X\nmid\widetilde{G}(X,Y)\).
If the height functions of \(F\) and \(G\) can be arranged into pairs of Lipschitz equivalent functions, then:

\begin{enumerate}[label=\roman*.]
\item \(e_F = e_G\)
\item For all \(t > 0\), we have
\(\widetilde{F}(tX,t^\beta Y) = t^{d-e}\widetilde{F}(X,Y)\)
\,and
\,\(\widetilde{G}(tX,t^\beta Y) = t^{d-e}\widetilde{G}(X,Y)\),
where \(e = e_F = e_G\).
\end{enumerate}
\end{proposition}

\begin{proof}
Let \(\beta = r/s\), where \(r > s > 0\) and \(\gcd(r,s) = 1\). Since \(F\) and \(G\) are \(\beta\)-quasihomogeneous polynomials of degree \(d\), we have:
\begin{equation*}
F(X,Y) = \sum_{k = 0}^m a_k X^{d - rk}Y^{sk}
\quad\text{and}\quad
G(X,Y) = \sum_{k = 0}^n b_k X^{d - rk}Y^{sk},
\end{equation*}
where \(a_m\neq 0\), \(b_n\neq 0\), and \(0\leq m,n\leq \lfloor d/r\rfloor\).
From these equations, it follows that \(e_F = d-rm\) and \(e_G = d-rn\). Thus, in order to prove that \(e_F = e_G\), it suffices to show that \(m = n\).

Let us prove that \(m = n\). Since \(f_+(t) = \sum_{k = 0}^m a_k t^{sk}\) and  
\(f_-(t) = \sum_{k = 0}^m (-1)^{d-rk}a_k t^{sk}\), we have 
\(\deg f_+ = \deg f_- = sm\); and since \(g_+(t) = \sum_{k = 0}^n b_k t^{sk}\) and  \(g_-(t) = \sum_{k = 0}^n (-1)^{d-rk}b_k t^{sk}\), we have 
\(\deg g_+ = \deg g_- = sn\). Also, by hypothesis, \(f_+\) and \(f_-\) are Lipschitz equivalent to \(g_+\) and \(g_-\) in some order. Hence, by Lemma \ref{lemma: bi-Lipschitz invariance of degree}, it follows that \(sm = sn\). Therefore, \(m = n\).

From now on, we drop the subscript and denote simply by \(e\) the multiplicity of \(X\) as a factor of either \(F\) or \(G\). Let us prove the second part of the proposition. Since \(F\) is a \(\beta\)-quasihomogeneous polynomial of degree \(d\), we have:
\begin{equation*}
F(tX,t^\beta Y) = t^d F(X,Y)= t^d X^e \widetilde{F}(X,Y).
\end{equation*}
On the other hand, since the multiplicity of \(X\) as a factor of \(F\) is equal to \(e\), we have:
\begin{equation*}
F(tX,t^\beta Y) = (tX)^e \widetilde{F}(tX, t^\beta Y)
= t^e X^e \widetilde{F}(tX,t^\beta Y).
\end{equation*}
From these equalities, it follows that
\begin{equation*}
\widetilde{F}(tX,t^\beta Y) = t^{d - e} \widetilde{F}(X,Y).
\end{equation*}
Clearly, the deduction above with \(F\) replaced with \(G\) yields the other identity.
\end{proof}

\begin{corollary}
\label{cor: beta-quasihomogeneous polynomials of the form cX^d}
Let \(F,G\in\R[X,Y]\) be \(\beta\)-quasihomogeneous polynomials of degree 
\(d\). If the height functions of \(F\) and \(G\) can be arranged into pairs of Lipschitz equivalent functions, then either both \(F\) and \(G\) are of the form \(cX^d\), or neither of them is.
\end{corollary}

\begin{proof}
Denote by \(e_F\) the multiplicity of \(X\) as a factor of \(F\), and by \(e_G\) the multiplicity of \(X\) as a factor of \(G\). By Proposition \ref{prop: cX^d in terms of the multiplicity of X as a factor}, \(F\) is of the form \(cX^d\) if and only if \(e_F = d\), and \(G\) is of the form \(cX^d\) if and only if \(e_G = d\). 
Suppose that the height functions of \(F\) and \(G\) can be arranged into pairs of Lipschitz equivalent functions. Then, by Proposition \ref{prop: invariance of the multiplicity of the factor X}, \(e_F = e_G\), so that \(e_F = d\) if and only if \(e_G = d\). Therefore \(F\) is of the form \(cX^d\) if and only if \(G\) is of the form \(cX^d\).
\end{proof}

The next proposition shows how to determine whether any two polynomials of the form \(cX^d\), with \(c\neq 0\) and \(d\geq 1\), are \({\cal R}\)-semialgebraically Lipschitz equivalent.

\begin{proposition}
\label{prop: Sufficient conditions for R-semialg. Lip. equivalence (special case)}
Let \(F(X,Y) = aX^d\) and \(G(X,Y) = bX^d\), where \(a,b\in\R\setminus\{0\}\) and \(d\geq 1\).
\begin{enumerate}[label = \roman*.]
\item If \(d\) is even, then \(F\) and \(G\) are \({\cal R}\)-semialgebraically Lipschitz equivalent if and only if \(a\) and \(b\) have the same sign.
\item If \(d\) is odd, then \(F\) and \(G\) are \({\cal R}\)-semialgebraically Lipschitz equivalent.
\end{enumerate}
\end{proposition}

\begin{proof}
(i) Suppose that \(d\) is even. If there exists a germ of semialgebraic bi-Lipschitz homeomorphism \(\Phi\from(\R^2,0)\to(\R^2,0)\) such that 
\(G\circ\Phi = F\) then \(b\cdot\Phi_1(x,y)^d = ax^d\) in a neighborhood of the origin, which implies that \(a\) and \(b\) have the same sign, since \(d\) is even. Now, assuming that \(a\) and \(b\) have the same sign, we have 
\(G\circ\Phi = F\), where \(\Phi(x,y) = \left(\left(\frac{a}{b}\right)^{\frac{1}{d}}\cdot x, y\right)\).

(ii) If \(d\) is odd, then \(G\circ\Phi = F\), where
\(\Phi(x,y) = \left(\left(\frac{a}{b}\right)^{\frac{1}{d}}\cdot x, y\right)\).
\end{proof}
\vspace{0.5cm}

Now, we turn our attention to the case where neither \(F\) nor \(G\) is of the form \(cX^d\).

\begin{lemma}
\label{lemma: limit condition, for e < d}
Let \(F,G\in\R[X,Y]\) be \(\beta\)-quasihomogeneous polynomials of degree 
\(d\) that are not of the form \(cX^d\). Suppose that the height functions \(f_+,f_-\) of\/ \(F\) and the height functions \(g_+,g_-\) of\/ \(G\) can be arranged into pairs of Lipschitz equivalent functions, so that
\((g_+,g_-)\circ(\lambda,\phi) = (f_+,f_-)\), where
\((\lambda,\phi) = ((\lambda_1,\lambda_2),(\phi_1,\phi_2))\)
is a zygothety. Then we have: 
\begin{equation}
\label{eq: limit condition, for e < d - in the statement of the lemma}
\abs{\lambda_1}^{\frac{d\beta}{d - e}}\cdot
   \lim_{\abs{t}\to+\infty} \abs{\frac{\phi_1(t)}{t}} = 
\abs{\lambda_2}^{\frac{d\beta}{d - e}}\cdot
   \lim_{\abs{t}\to+\infty} \abs{\frac{\phi_2(t)}{t}},
\end{equation}
where \(e\) is the multiplicity of \(X\) as a factor of \(F\), and also as a factor of \(G\).\footnote{By hypothesis, the height functions of \/\(F\) and \/\(G\) can be arranged into pairs of Lipschitz equivalent functions. Thus, by Proposition \ref{prop: invariance of the multiplicity of the factor X}, the multiplicity of \(X\) as a factor of \(F\) is equal to the multiplicity of \(X\) as a factor of \(G\). Also, since neither \(F\) nor \(G\) is of the form \(cX^d\), the height functions \(f_+\), \(f_-\), \(g_+\), and \(g_-\) are all nonconstant. Hence, by Lemma \ref{lemma: analytic extension of phi(t)/t}, the limit \(\lim_{\abs{t}\to+\infty} \phi_i(t)/t\) is a well-defined nonzero real number for \(i = 1,2\).}
\end{lemma}

\begin{proof}
We begin with a preliminary remark.  Given that \(e\) is the multiplicity of \(X\) as a factor of \(F\), and also as a factor of \(G\), it follows that \(F(X,Y) = X^{e}\cdot\widetilde{F}(X,Y)\) and \(G(X,Y) = X^{e}\cdot\widetilde{G}(X,Y)\),
where \(X\nmid\widetilde{F}(X,Y)\) and \(X\nmid\widetilde{G}(X,Y)\).
By Proposition \ref{prop: cX^d in terms of the multiplicity of X as a factor}, we have \(e\leq d\), with equality if and only if \(F\) and \(G\) are of the form 
\(cX^d\). By hypothesis, \(F\) and \(G\) are not of this form; therefore, we have \(e < d\).

Now, we proceed to the proof of (\ref{eq: limit condition, for e < d - in the statement of the lemma}).
First, we consider the case where \(\lambda > 0\). In this case, we have
\begin{equation*}
\abs{\lambda_1}^d\cdot g_+\circ\phi_1 = f_+
\quad\text{and}\quad
\abs{\lambda_2}^d\cdot g_-\circ\phi_2 = f_-.
\end{equation*}
From the identity
\(\abs{\lambda_1}^d\cdot g_+\circ\phi_1 = f_+\),
it follows naturally that 
\begin{equation*}
\abs{\lambda_1}^d\cdot \widetilde{G}\left(t^{-\frac{1}{\beta}},\frac{\phi_1(t)}{t}\right) = \widetilde{F}\left(t^{-\frac{1}{\beta}},1\right),
\text{ for \(t > 0\)}.
\end{equation*}
In this derivation, we use the fact that \(\widetilde{F}\) and \(\widetilde{G}\) are \(\beta\)-quasihomogeneous of the same degree (by the second part of Proposition \ref{prop: invariance of the multiplicity of the factor X}, both \(\widetilde{F}\) and \(\widetilde{G}\) are \(\beta\)-quasihomogeneous of degree \(d - e\)).
Letting \(t\to+\infty\), we obtain:
\begin{equation*}
\abs{\lambda_1}^d\cdot
\widetilde{G}\left(0,\lim_{\abs{t}\to+\infty} \frac{\phi_1(t)}{t}\right) = 
\widetilde{F}(0,1).
\end{equation*}
Since, \(\widetilde{G}\) is \(\beta\)-quasihomogeneous of degree \(d - e\), it follows that
\begin{equation}
\label{eq: (lambda_1)^d G tilde (...) = F tilde (0,1)}
\widetilde{G}\left(0, 
   \abs{\lambda_1}^{\frac{d\beta}{d-e}} \cdot 
      \lim_{\abs{t}\to+\infty} \frac{\phi_1(t)}{t}\right) 
= \widetilde{F}(0,1).
\end{equation}
Similarly, from the identity \(\abs{\lambda_2}^d\cdot g_-\circ\phi_2 = f_-\), we obtain:
\begin{equation}
\label{eq: (lambda_2)^d G tilde (...) = F tilde (0,1)}
\widetilde{G}\left(0, 
   \abs{\lambda_2}^{\frac{d\beta}{d-e}} \cdot 
      \lim_{\abs{t}\to+\infty} \frac{\phi_2(t)}{t}\right) 
= \widetilde{F}(0,1).
\end{equation}
From (\ref{eq: (lambda_1)^d G tilde (...) = F tilde (0,1)}) and 
(\ref{eq: (lambda_2)^d G tilde (...) = F tilde (0,1)}), we see that
\begin{equation}
\label{eq: G tilde(0,lambda_i^{d beta / d - e} phi_i(t)/t}
\widetilde{G}\left(0, 
   \abs{\lambda_1}^{\frac{d\beta}{d-e}} \cdot 
      \lim_{\abs{t}\to+\infty} \frac{\phi_1(t)}{t}\right) 
=
\widetilde{G}\left(0, 
   \abs{\lambda_2}^{\frac{d\beta}{d-e}} \cdot 
      \lim_{\abs{t}\to+\infty} \frac{\phi_2(t)}{t}\right). 
\end{equation}
Since \(G\) is \(\beta\)-quasihomogeneous of degree \(d\), we have
\(G(X,Y) = \sum_{k = 0}^n b_k X^{d - rk}Y^{sk}\), where \(b_n\neq 0\) and \(0\leq n\leq \lfloor d/r\rfloor\).
Then it follows that \(\widetilde{G}(X,Y) = \sum_{k = 0}^nb_kX^{r(n-k)}Y^{sk}\). Using this identity, we can rewrite equation (\ref{eq: G tilde(0,lambda_i^{d beta / d - e} phi_i(t)/t}) in the form:
\begin{equation*}
b_n\cdot\left(
   \abs{\lambda_1}^{\frac{d\beta}{d - e}}\cdot
      \lim_{\abs{t}\to+\infty} \frac{\phi_1(t)}{t}
\right)^{sn}
=
b_n\cdot\left(
   \abs{\lambda_2}^{\frac{d\beta}{d - e}}\cdot
      \lim_{\abs{t}\to+\infty} \frac{\phi_2(t)}{t}
\right)^{sn}.
\end{equation*}
And from this, we obtain the identity:
\begin{equation*}
\abs{\lambda_1}^{\frac{d\beta}{d - e}}\cdot
      \lim_{\abs{t}\to+\infty} \abs{\frac{\phi_1(t)}{t}}
=
\abs{\lambda_2}^{\frac{d\beta}{d - e}}\cdot
      \lim_{\abs{t}\to+\infty} \abs{\frac{\phi_2(t)}{t}}.
\end{equation*}

In the case where \(\lambda < 0\), we have
\begin{equation*}
\abs{\lambda_1}^d\cdot g_-\circ\phi_1 = f_+
\quad\text{and}\quad
\abs{\lambda_2}^d\cdot g_+\circ\phi_2 = f_-.
\end{equation*}
Proceeding along the same lines as we did in the case where \(\lambda >0\), we can show that 
\begin{equation*}
\widetilde{G}\left(0, 
   \abs{\lambda_1}^{\frac{d\beta}{d-e}} \cdot 
      \lim_{\abs{t}\to+\infty} \frac{\phi_1(t)}{t}\right) 
= (-1)^e\cdot\widetilde{F}(0,1)
= \widetilde{G}\left(0, 
   \abs{\lambda_2}^{\frac{d\beta}{d-e}} \cdot 
      \lim_{\abs{t}\to+\infty} \frac{\phi_2(t)}{t}\right).
\end{equation*}
So equation (\ref{eq: G tilde(0,lambda_i^{d beta / d - e} phi_i(t)/t}) also holds in the case where \(\lambda < 0\), and from this point we can repeat the argument used in the case where \(\lambda > 0\) to complete the proof.
\end{proof}

\begin{proposition}
\label{prop: characterization of beta-transitions}
Let \(F,G\in\R[X,Y]\) be \(\beta\)-quasihomogeneous polynomials of degree \(d\) that are not of the form \(c X^d\), and let \(f_+,f_-\) be the height functions of\/ \(F\) and \(g_+,g_-\) be the height functions of\/ \(G\). Suppose that \((g_+,g_-)\circ(\lambda,\phi) = (f_+,f_-)\) for some zygothety \((\lambda,\phi)\). Then \((\lambda,\phi)\) is \(\beta\)-regular if and only if the following conditions hold:
\begin{enumerate}[label=\roman*.]
\item \(\phi_1\) and \(\phi_2\) are {\it coherent}.\footnote{We say that two monotone functions \(\phi_1,\phi_2\from\R\to\R\) are {\it coherent} if they are either both increasing or both decreasing.}
\item Neither of the polynomials \(F,G\) has \(X\) as a factor, or \(\lambda_1 = \lambda_2\).
\end{enumerate}
\end{proposition}

\begin{proof}
First, suppose that \((\lambda,\phi)\) is \(\beta\)-regular. Then the limits\/ 
\(\lim_{\abs{t}\to+\infty} \phi_1(t)/t\) 
\,and 
\,\(\lim_{\abs{t}\to+\infty} \phi_2(t)/t\)
exist and are nonzero, and we have 
\begin{equation}
\label{eq: limit condition for beta-transitions}
\abs{\lambda_1}^\beta\cdot\lim_{\abs{t}\to+\infty}\frac{\phi_1(t)}{t}
=
\abs{\lambda_2}^\beta\cdot\lim_{\abs{t}\to+\infty}\frac{\phi_2(t)}{t}.
\end{equation}
Since \(\abs{\lambda_1} > 0\) and \(\abs{\lambda_2} > 0\), it follows that 
\(\lim_{\abs{t}\to+\infty} \phi_1(t)/t\) and \(\lim_{\abs{t}\to+\infty} \phi_2(t)/t\) have the same sign. And since \(\phi_1\) and \(\phi_2\) are monotone, this implies that they are coherent. Hence, condition (i) is satisfied.

Now, since \((g_+,g_-)\circ(\lambda,\phi) = (f_+,f_-)\), where \((\lambda,\phi)\) is a zygothety, the multiplicity of \(X\) as a factor of \(F\) is equal to the multiplicity of \(X\) as a factor of \(G\) (see the first part of Proposition \ref{prop: invariance of the multiplicity of the factor X}). Let us denote by \(e\) the multiplicity of \(X\) both as a factor of \(F\) and as a factor of \(G\). 
Then, by Lemma \ref{lemma: limit condition, for e < d}, we have:
\begin{equation}
\label{eq: limit condition, for e < d}
\abs{\lambda_1}^{\frac{d\beta}{d - e}}
   \lim_{\abs{t}\to+\infty} \abs{\frac{\phi_1(t)}{t}} = 
\abs{\lambda_2}^{\frac{d\beta}{d - e}}
   \lim_{\abs{t}\to+\infty} \abs{\frac{\phi_2(t)}{t}}.
\end{equation}
Since the limits \(\lim_{\abs{t}\to+\infty} \phi_1(t)/t\) and 
\(\lim_{\abs{t}\to+\infty} \phi_2(t)/t\) have the same sign, it follows that
\begin{equation}
\label{eq: limit condition, for e < d - without absolute value}
\abs{\lambda_1}^{\frac{d\beta}{d - e}}
   \lim_{\abs{t}\to+\infty} \frac{\phi_1(t)}{t} = 
\abs{\lambda_2}^{\frac{d\beta}{d - e}}
   \lim_{\abs{t}\to+\infty} \frac{\phi_2(t)}{t}.
\end{equation}
Since the limits \(\lim_{\abs{t}\to+\infty} \phi_1(t)/t\) and 
\(\lim_{\abs{t}\to+\infty} \phi_2(t)/t\) are both nonzero, it follows from equations 
(\ref{eq: limit condition for beta-transitions}) and (\ref{eq: limit condition, for e < d - without absolute value}) that
\begin{equation}
\label{eq: ratio of powers of abs(lambda)}
\frac{\abs{\lambda_1}^{\frac{d\beta}{d-e}}}{\abs{\lambda_1}^\beta} = 
\frac{\abs{\lambda_2}^{\frac{d\beta}{d-e}}}{\abs{\lambda_2}^\beta}.
\end{equation}
Equivalently,
\begin{equation}
\label{eq: simplified equation powers of lambda}
\abs{\lambda_1}^\frac{e\beta}{d-e} = \abs{\lambda_2}^\frac{e\beta}{d-e}.
\end{equation}
Furthermore, this equality holds if and only if \(e = 0\) or 
\(\abs{\lambda_1} = \abs{\lambda_2}\); and this is equivalent to condition (ii) because \(\lambda_1\) and \(\lambda_2\) have the same sign. Since (\ref{eq: simplified equation powers of lambda}) actually holds, condition (ii) is satisfied.

Now, in order to prove the converse, suppose that conditions (i) and (ii) hold. By hypothesis, neither \(F\) nor \(G\) is of the form \(cX^d\) and  \((g_+,g_-)\circ(\lambda,\phi) = (f_+,f_-)\), where \((\lambda,\phi)\) is a zygothety. Hence, by Lemma \ref{lemma: analytic extension of phi(t)/t}, the limits\/ \(\lim_{\abs{t}\to+\infty} \phi_1(t)/t\) \,and\,\(\lim_{\abs{t}\to+\infty} \phi_2(t)/t\) exist and are nonzero. Also, by Lemma \ref{lemma: limit condition, for e < d}, it follows from our hypotheses that (\ref{eq: limit condition, for e < d}) still holds for this part of the argument. This, along with condition (i), implies (\ref{eq: limit condition, for e < d - without absolute value}). On the other hand, as we have already proved, condition (ii) is equivalent to (\ref{eq: ratio of powers of abs(lambda)}). Since we are assuming that condition (ii) is satisfied, (\ref{eq: ratio of powers of abs(lambda)}) holds. From (\ref{eq: limit condition, for e < d - without absolute value}) and (\ref{eq: ratio of powers of abs(lambda)}), we obtain (\ref{eq: limit condition for beta-transitions}). Therefore, \((\lambda,\phi)\) is \(\beta\)-regular.
\end{proof}

\begin{corollary}
\label{cor: Consequences of the characterization of beta-transitions}
Let \(F,G\in\R[X,Y]\) be \(\beta\)-quasihomogeneous polynomials of degree \(d\) that are not of the form \(c X^d\), and let \(f_+,f_-\) be the height functions of \(F\) and \(g_+,g_-\) be the height functions of \(G\). 
Also, let \(\beta = r/s\), where \(r > s > 0\) and \(\gcd(r,s) = 1\).
Suppose that \((g_+,g_-)\circ(\lambda,\phi) = (f_+,f_-)\) for some zygothety \((\lambda,\phi)\). Then the following holds:
\begin{enumerate}[label = \textnormal{(\alph*)}]
\item If\/ \(r\) is even or \(s\) is odd, then there exists a \(\beta\)-regular zygothety \((\tilde{\lambda},\tilde{\phi})\) such that
\((g_+,g_-)\circ(\tilde{\lambda},\tilde{\phi}) = (f_+,f_-)\).
\item If\/ \(s\) is even, then there exists 
\(\tilde{\phi} = (\tilde{\phi}_1,\tilde{\phi}_2)\),
with \(\tilde{\phi}_1\) and \(\tilde{\phi}_2\) coherent, such that
\((g_+,g_-)\circ(\lambda,\tilde{\phi}) = (f_+,f_-)\).
\end{enumerate}
\end{corollary}

\begin{proof}
Let \(F(X,Y) = \sum_{k=0}^n a_kX^{d-rk}Y^{sk}\) and 
\(G(X,Y) = \sum_{k=0}^n b_k X^{d-rk}Y^{sk}\), with \(a_n,b_n\neq 0\), 
\(n\geq 1\). (In the proof of Proposition \ref{prop: invariance of the multiplicity of the factor X}, we showed that the upper limit of summation \(n\) is the same for \(F\) and \(G\), provided that the height functions of \(F\) and \(G\) can be arranged into pairs of Lipschitz equivalent functions --- which is the case, since
\((g_+,g_-)\circ(\lambda,\phi) = (f_+,f_-)\) for some zygothety
\((\lambda,\phi)\). Also, we have \(n\geq 1\) because neither \(F\) nor \(G\) is of the form \(cX^d\).) Let us proceed to the proof of parts (a) and (b).
\begin{enumerate}[label=(\alph*)]
\item {\bf Case 1.} \(r\) is even

In this case, we have
\(f_-(t) = (-1)^d\cdot f_+(t)\) and \(g_-(t) = (-1)^d\cdot g_+(t)\).
By hypothesis, there exists a zygothety \((\lambda,\phi)\) such that 
\((g_+,g_-)\circ(\lambda,\phi) = (f_+,f_-)\). We claim that 
\((g_+,g_-)\circ(\tilde\lambda,\tilde\phi) = (f_+,f_-)\), where 
\(\tilde\lambda = (\lambda_1, \lambda_1)\) and \(\tilde\phi = (\phi_1,\phi_1)\).
In fact, if \(\lambda > 0\) then \(\abs{\lambda_1}^d\cdot g_+\circ\phi_1 = f_+\) and hence 
\begin{equation*}
\abs{\lambda_1}^d\cdot g_-(\phi_1(t))  
= (-1)^d\cdot \abs{\lambda_1}^d\cdot g_+(\phi_1(t))
= (-1)^d\cdot f_+(t)
= f_-(t),
\end{equation*}
so we also have \(\abs{\lambda_1}^d\cdot g_-\circ\phi_1 = f_-\).
If \(\lambda < 0\) then \(\abs{\lambda_1}^d\cdot g_-\circ\phi_1 = f_+\) and hence
\begin{equation*}
\abs{\lambda_1}^d\cdot g_+(\phi_1(t))  
= (-1)^d\cdot \abs{\lambda_1}^d\cdot g_-(\phi_1(t))
= (-1)^d\cdot f_+(t)
= f_-(t),
\end{equation*}
so we also have \(\abs{\lambda_1}^d\cdot g_+\circ\phi_1 = f_-\).
By Proposition \ref{prop: characterization of beta-transitions}, 
\((\tilde\lambda,\tilde\phi)\) is a \(\beta\)-regular zygothety.\\

\noindent
{\bf Case 2.} \(r\) and \(s\) are both odd

In this case, we have \(f_-(t) = (-1)^d\cdot f_+(-t)\) and \(g_-(t) = (-1)^d\cdot g_+(-t)\). By hypothesis, there exists a zygothety \((\lambda,\phi)\) such that 
\((g_+,g_-)\circ(\lambda,\phi) = (f_+,f_-)\). We claim that 
\((g_+,g_-)\circ(\tilde\lambda,\tilde\phi) = (f_+,f_-)\), where 
\(\tilde\lambda = (\lambda_1, \lambda_1)\) and 
\(\tilde\phi(t) = (\phi_1(t),-\phi_1(-t))\).
In fact, if \(\lambda > 0\) then \(\abs{\lambda_1}^d\cdot g_+\circ\phi_1 = f_+\) and hence 
\begin{equation*}
\abs{\lambda_1}^d\cdot g_-(-\phi_1(-t))  
= (-1)^d\cdot \abs{\lambda_1}^d\cdot g_+(\phi_1(-t))
= (-1)^d\cdot f_+(-t)
= f_-(t),
\end{equation*}
so we also have \(\abs{\lambda_1}^d\cdot g_-(-\phi_1(-t)) = f_-(t)\).
If \(\lambda < 0\) then \(\abs{\lambda_1}^d\cdot g_-\circ\phi_1 = f_+\) and hence
\begin{equation*}
\abs{\lambda_1}^d\cdot g_+(-\phi_1(-t))  
= (-1)^d\cdot \abs{\lambda_1}^d\cdot g_-(\phi_1(-t))
= (-1)^d\cdot f_+(-t)
= f_-(t),
\end{equation*}
so we also have \(\abs{\lambda_1}^d\cdot g_+(-\phi_1(-t)) = f_-(t)\).
By Proposition \ref{prop: characterization of beta-transitions}, 
\((\tilde\lambda,\tilde\phi)\) is a \(\beta\)-regular zygothety.

\item Suppose that \(s\) is even. In this case, we have
\(g_+(-t) = g_+(t)\) and \(g_-(-t) = g_-(t)\).
By hypothesis, there exists a zygothety \((\lambda,\phi)\) such that
\((g_+,g_-)\circ (\lambda,\phi) = (f_+,f_-)\). We claim that 
\((g_+,g_-)\circ (\lambda,\overline{\phi}) = (f_+,f_-)\), where
\(\overline{\phi} = (\phi_1,-\phi_2)\). In fact, if \(\lambda > 0\) then 
\(\abs{\lambda_2}^d\cdot g_-\circ \phi_2 = f_-\) and hence also
\(\abs{\lambda_2}^d\cdot g_-\circ (-\phi_2) = f_-\). If \(\lambda < 0\) then 
\(\abs{\lambda_2}^d\cdot g_+\circ \phi_2 = f_-\) and hence also
\(\abs{\lambda_2}^d\cdot g_+\circ (-\phi_2) = f_-\). 
Finally, notice that \(\phi_1\) is coherent with either \(\phi_2\) or \(-\phi_2\). If \(\phi_1\) and \(\phi_2\) are coherent, we take \(\tilde\phi = \phi\). If \(\phi_1\) and \(-\phi_2\) are coherent, we take \(\tilde\phi = \overline{\phi}\).
\vspace{-20pt}
\end{enumerate}
\end{proof}

\begin{theorem}
\label{thm: Sufficient conditions for R-semialg. Lip. equivalence}
Let \(F,G\in\R[X,Y]\) be \(\beta\)-quasihomogeneous polynomials of degree \(d\) that are not of the form \(cX^d\), and let \(f_+,f_-\) be the height functions of\/ \(F\) and \(g_+,g_-\) be the height functions of\/ \(G\). Also, let \(\beta = r/s\), where \(r > s > 0\) and 
\(\gcd(r,s) = 1\). Suppose that
\((g_+,g_-)\circ(\lambda,\phi) = (f_+,f_-)\)
for some zygothety \((\lambda,\phi) = ((\lambda_1,\lambda_2),(\phi_1,\phi_2))\). If any of the following conditions is satisfied, then\/ \(F\) and\/ \(G\) are
\({\cal R}\)-semialgebraically Lipschitz equivalent:
\begin{enumerate}[label=\textnormal{(\alph*)}]
\item \(r\) is even or \(s\) is odd.
\item \(\lambda_1 = \lambda_2\)
\item Neither \(F\) nor \(G\) has \(X\) as a factor.
\end{enumerate}
\end{theorem}

\begin{proof}
If \(r\) is even or \(s\) is odd then, by Corollary \ref{cor: Consequences of the characterization of beta-transitions}, there exists a \(\beta\)-regular zygothety
\((\tilde\lambda,\tilde\phi)\) such that
\((g_+,g_-)\circ(\tilde\lambda,\tilde\phi) = (f_+,f_-)\). 
Hence, by Theorem \ref{thm: beta-regular zygothety, R-semialgebraic Lipschitz equivalence}, \(F\) and \(G\) are \({\cal R}\)-semialgebraically Lipschitz equivalent.

Now, assume that either (b) or (c) holds. If \(s\) is odd, then condition (a) is satisfied, and therefore \(F\) and \(G\) are \({\cal R}\)-semialgebraically Lipschitz equivalent, as we have just proved. If \(s\) is even then, by Corollary \ref{cor: Consequences of the characterization of beta-transitions}, there exists \(\tilde\phi = (\tilde\phi_1,\tilde\phi_2)\), with \(\tilde\phi_1\) and 
\(\tilde\phi_2\) coherent, such that 
\((g_+,g_-)\circ(\lambda,\tilde\phi) = (f_+,f_-)\). 
Since we are assuming that either (b) or (c) holds, Proposition \ref{prop: characterization of beta-transitions} guarantees that 
\((\lambda,\tilde\phi)\) is \(\beta\)-regular. 
Then, by Theorem \ref{thm: beta-regular zygothety, R-semialgebraic Lipschitz equivalence}, it follows that \(F\) and \(G\) are \({\cal R}\)-semialgebraically Lipschitz equivalent.
\end{proof}

\begin{corollary}
\label{cor: Sufficient conditions for R-semialg. Lip. equivalence}
Let\/ \(F,G\in\R[X,Y]\) be \(\beta\)-quasihomogeneous polynomials of degree \(d\) that are not of the form \(cX^d\). Suppose that the height functions \(f_+,f_-\) of\/ \(F\) and the height functions 
\(g_+,g_-\) of\/ \(G\) can be arranged into pairs of Lipschitz equivalent functions.
If one of the height functions\/ \(f_+,f_-,g_+,g_-\) has no critical points, then\/ \(F\) and\/ \(G\) are \({\cal R}\)-semialgebraically Lipschitz equivalent.
\end{corollary}

\begin{proof}
By hypothesis, the height functions of \(F\) and \(G\) can be arranged into pairs of Lipschitz equivalent functions. Hence, by Corollary \ref{cor: condition on height functions in terms of group action}, there exists a zygothety
\((\lambda,\phi) = ((\lambda_1,\lambda_2),(\phi_1,\phi_2))\) such that 
\((g_+,g_-)\circ(\lambda,\phi) = (f_+,f_-)\). From this, and assuming that one of the height functions\/ \(f_+,f_-,g_+,g_-\) has no critical points, we prove that there exists a zygothety \((\tilde\lambda,\tilde\phi) = ((\tilde\lambda_1,\tilde\lambda_2),(\tilde\phi_1,\tilde\phi_2))\),
with \(\tilde\lambda_1 = \tilde\lambda_2\), such that
\((g_+,g_-)\circ(\tilde\lambda,\tilde\phi) = (f_+,f_-)\).
Then it follows, by Theorem \ref{thm: Sufficient conditions for R-semialg. Lip. equivalence}, that \(F\) and \(G\) are \({\cal R}\)-semialgebraically Lipschitz equivalent.

We consider only the case where \(\lambda > 0\); the case where \(\lambda < 0\) is analogous. In this case, we have
\(\abs{\lambda_1}^d g_+\circ\phi_1 = f_+\)
and
\(\abs{\lambda_2}^d g_-\circ\phi_2 = f_-\).
Since \(f_+\cong g_+\) and \(f_-\cong g_-\), we see that \(f_+\) and \(g_+\) have the same number of critical points, as do \(f_-\) and \(g_-\). Thus, we only need to consider the following possibilities: (A) Both \(f_+\) and \(g_+\) have no critical points. (B) Both \(f_-\) and \(g_-\) have no critical points.

Suppose that both \(f_+\) and \(g_+\) have no critical points.
The proof of Theorem \ref{thm: no critical points} shows that
if \(f,g\from\R\to\R\) are polynomial functions of degree
\(d\geq 1\) with no critical points, then there exists a bi-Lipschitz homeomorphism \(\phi\from\R\to\R\) such that \(g\circ\phi = f\). By applying this to \(f_+\) and \(\abs{\lambda_2}^d g_+\), we obtain a bi-Lipschitz homeomorphism \(\tilde{\phi}_1\from\R\to\R\) such that
\(\abs{\lambda_2}^d g_+ \circ \tilde{\phi}_1 = f_+\). Thus, we have
\((g_+,g_-)\circ(\tilde\lambda,\tilde\phi) = (f_+,f_-)\), where 
\(\tilde\lambda = (\lambda_2,\lambda_2)\) and 
\(\tilde\phi = (\tilde{\phi}_1,\phi_2)\).
Similarly, if both \(f_-\) and \(g_-\) have no critical points, we can obtain a bi-Lipschitz homeomorphism \(\tilde{\phi}_2\from\R\to\R\) such that
\(\abs{\lambda_1}^d g_- \circ \tilde{\phi}_2 = f_-\), and then we have 
\((g_+,g_-)\circ(\tilde\lambda,\tilde\phi) = (f_+,f_-)\), where 
\(\tilde\lambda = (\lambda_1,\lambda_1)\) and 
\(\tilde\phi = (\phi_1,\tilde{\phi}_2)\).
\end{proof}

\begin{lemma}
\label{lemma: r odd, s even}
Let \(F,G\in\R[X,Y]\) be \(\beta\)-quasihomogeneous polynomials of degree \(d\) that are not of the form \(cX^d\), and let \(f_+,f_-\) be the height functions of \(F\) and \(g_+,g_-\) be the height functions of \(G\). Also, let \(\beta = r/s\), where \(r > s > 0\) and 
\(\gcd(r,s) = 1\). Suppose that
\((g_+,g_-)\circ(\lambda,\phi) = (f_+,f_-)\)
for some zygothety \((\lambda,\phi) = ((\lambda_1,\lambda_2),(\phi_1,\phi_2))\). If \(r\) is odd and \(s\) is even, then:
\begin{enumerate}[label=\roman*.]
\item Either both of the polynomials \(F,G\) have \(Y\) as a factor, or neither has \(Y\) as a factor.
\item If neither \(F\) nor \(G\) has \(Y\) as a factor then \(\lambda_1 = \lambda_2\).
\item If\/ \(Y\) is a factor of both \(F\) and \(G\), and one of the height functions
\(f_+,f_-,g_+,g_-\) has only one critical point, then there exists a zygothety \((\tilde\lambda,\tilde\phi)=
((\tilde\lambda_1,\tilde\lambda_2),(\tilde\phi_1,\tilde\phi_2))\),
with \(\tilde\lambda_1=\tilde\lambda_2\), such that
\((g_+,g_-)\circ(\tilde\lambda,\tilde\phi)=(f_+,f_-)\).
\end{enumerate}
\end{lemma}

\begin{proof}
Let \(F(X,Y)=\sum_{k=0}^n a_k X^{d-rk}Y^{sk}\) and 
\(G(X,Y) = \sum_{k=0}^n b_k X^{d-rk}Y^{sk}\), 
with \(a_n,b_n\neq 0\), \(n\geq 1\). (In the proof of 
Proposition \ref{prop: invariance of the multiplicity of the factor X},
we showed that the upper limit of the summation \(n\) is the same for \(F\) and \(G\), provided that the height functions of \(F\) and \(G\) can be arranged into pairs of Lipschitz equivalent functions. Also, we have \(n\geq 1\) because neither \(F\) nor \(G\) is of the form \(cX^d\).)

Suppose that \(r\) is odd and \(s\) is even. Then we have:
\(f_+(t) = \sum_{k=0}^n a_kt^{sk},\,
f_-(t) = (-1)^d\sum_{k=0}^n(-1)^ka_kt^{sk}\),\,
\(g_+(t) = \sum_{k=0}^n b_kt^{sk}\), and 
\(g_-(t) = (-1)^d\sum_{k=0}^n(-1)^kb_kt^{sk}\).

We claim that
\begin{equation}
\label{eq: lambda_i b_0 = a_0}
\abs{\lambda_1}^d b_0 = a_0 
\quad\text{and}\quad 
\abs{\lambda_2}^d b_0 = a_0.
\end{equation}
In order to prove this, we start by noting that \(f_+,f_-,g_+,g_-\) are all even functions, because \(s\) is even. Thus, for each of these height functions, \(0\) is a critical point and the number of negative critical points is equal to the number of positive critical points (more precisely, the map \(t\mapsto -t\) establishes a 1-1 correspondence between the positive critical points and the negative critical points). Also, we note that each of the functions \(f_+,f_-,g_+,g_-\) has only a finite number of critical points, since they are all nonconstant polynomial functions (each of them has degree \(sn\) because \(a_n,b_n\neq 0\)).
Now, suppose that \(\lambda > 0\). Then
\(\abs{\lambda_1}^d g_+\circ\phi_1 = f_+\)
and
\(\abs{\lambda_2}^d g_-\circ\phi_2 = f_-\).
Let \(-t_p < \dots < -t_1 < 0 < t_1 < \dots < t_p\) be the critical points of \(f_+\), and let \(-s_p < \dots < -s_1 < 0 < s_1 < \dots < s_p\) be the critical points of 
\(g_+\) (\(f_+\) and \(g_+\) have the same number of critical points because they are Lipschitz equivalent). Since \(\phi_1\) is monotone, injective, and takes critical points of \(f_+\) to critical points of \(g_+\), it follows that \(\phi_1(0) = 0\). Consequently, since 
\(\abs{\lambda_1}^d g_+\circ\phi_1 = f_+\), we have 
\(\abs{\lambda_1}^d g_+(0) = f_+(0)\). Equivalently, since 
\(f_+(0) = a_0\) and \(g_+(0) = b_0\), we have 
\(\abs{\lambda_1}^d b_0 = a_0\). Similarly, we can show that \(\phi_2(0) = 0\), and then we can use this along with the equation \(\abs{\lambda_2}^d g_-\circ\phi_2 = f_-\) to conclude that \(\abs{\lambda_2}^d b_0 = a_0\).
This shows that (\ref{eq: lambda_i b_0 = a_0}) holds for \(\lambda > 0\). The proof for \(\lambda < 0\) is analogous.

Now, we proceed to the proof of the proposition itself. 
From (\ref{eq: lambda_i b_0 = a_0}), we see that either \(a_0 = b_0 = 0\),
or \(a_0\neq 0\) and \(b_0\neq 0\). Clearly, \(Y\) is a factor of \(F\) if and only if \(a_0 = 0\), and \(Y\) is a factor of \(G\) if and only if \(b_0 = 0\). Hence, either both \(F\) and \(G\) have \(Y\) as a factor, or neither has \(Y\) as a factor. This proves the first part of the proposition. 

For the second part, suppose that neither \(F\) nor \(G\) has \(Y\) as a factor, so that \(a_0\neq 0\) and \(b_0\neq 0 \). Then, from (\ref{eq: lambda_i b_0 = a_0}), it follows that \(\abs{\lambda_1} = \abs{\lambda_2}\). Since \(\lambda_1\) and \(\lambda_2\) have the same sign, we actually have \(\lambda_1 = \lambda_2\). 

Now, we prove the third part. Suppose that \(Y\) is a factor of both \(F\) and 
\(G\), and that one of the height functions \(f_+,f_-,g_+,g_-\) has only one critical point. Let us consider the case where \(\lambda > 0\). In this case, we have:
\begin{equation*}
\abs{\lambda_1}^d g_+\circ\phi_1 = f_+
\quad\text{and}\quad
\abs{\lambda_2}^d g_-\circ\phi_2 = f_-\,.
\end{equation*}
Since \(f_+\cong g_+\) and \(f_-\cong g_-\), we see that \(f_+\) and \(g_+\) have the same number of critical points, and also that \(f_-\) and \(g_-\) have the same number of critical points. Thus, we only need to consider the following possibilities: (A) Both \(f_+\) and \(g_+\) have only one critical point. (B) Both \(f_-\) and \(g_-\) have only one critical point. 

Suppose that both \(f_+\) and \(g_+\) have only one critical point. We have already seen that \(0\) is a critical point of both \(f_+\) and \(g_+\). Hence, \(0\) is the only critical point of \(f_+\), and also the only critical point of \(g_+\). Recall that \(f_+(0) = a_0\) and \(g_+(0) = b_0\). Since, by hypothesis, \(Y\) is a factor of both \(F\) and \(G\), we have \(a_0=b_0=0\). Therefore, 
\(f_+(0)=g_+(0)=0\). From the proof of Theorem \ref{thm: only one critical point}, it is clear that if \(f,g\from\R\to\R\) are Lipschitz equivalent polynomial functions of degree \(\geq 1\) such that \(f\) has only one critical point \(t_0\) and \(g\) has only one critical point \(s_0\), then given a constant \(c>0\) such that \(g_+(s_0) = c f_+(t_0)\), there exists a bi-Lipschitz homeomorphism 
\(\phi\from\R\to\R\) such that \(g\circ\phi = c f\). In particular, if \(g(s_0) = f(t_0) = 0\), then for any constant \(c > 0\), there exists a bi-Lipschitz homeomorphism \(\phi\from\R\to\R\) such that \(g\circ\phi = c f\). By applying this result to \(f_+\) and \(g_+\), we obtain a bi-Lipschitz homeomorphism \(\tilde\phi_1\from\R\to\R\) such that
\(\abs{\lambda_2}^d g_+\circ\tilde\phi_1 = f_+\). Thus, we have
\((g_+,g_-)\circ(\tilde\lambda,\tilde\phi) = (f_+,f_-)\), where
\(\tilde\lambda = (\lambda_2,\lambda_2)\) and 
\(\tilde\phi = (\tilde\phi_1,\phi_2)\).

Similarly, if both \(f_-\) and \(g_-\) have only one critical point then we can obtain a bi-Lipschitz homeomorphism  \(\tilde\phi_2\from\R\to\R\) such that \(\abs{\lambda_1}^d g_-\circ\tilde\phi_2 = f_-\), whence
\((g_+,g_-)\circ(\tilde\lambda,\tilde\phi) = (f_+,f_-)\), where
\(\tilde\lambda = (\lambda_1,\lambda_1)\) and 
\(\tilde\phi = (\phi_1,\tilde\phi_2)\). This proves the third part of the proposition for \(\lambda > 0\). The proof for \(\lambda < 0\) is analogous.
\end{proof}

\begin{corollary}
\label{cor: r odd, s even}
Let\/ \(F,G\in\R[X,Y]\) be \(\beta\)-quasihomogeneous polynomials of degree 
\(d\) that are not of the form \(cX^d\). Suppose that the height functions \(f_+,f_-\) of\/ \(F\) and the height functions \(g_+,g_-\) of\/ \(G\) can be arranged into pairs of Lipschitz equivalent functions.
If any of the following conditions is satisfied, then\/ \(F\) and\/ \(G\) are 
\({\cal R}\)-semialgebraically Lipschitz equivalent:
\begin{enumerate}[label=\textnormal{(\alph*)}]
\item Neither \(F\) nor \(G\) has\/ \(Y\) as a factor.
\item One of the height functions\/ \(f_+,f_-,g_+,g_-\) has only one critical point.
\end{enumerate}
\end{corollary}

\begin{proof}
Let \(\beta = r/s\), where \(r > s > 0\) and \(\gcd(r,s) = 1\).
By hypothesis, the height functions of \(F\) and \(G\) can be arranged into pairs of Lipschitz equivalent functions. Hence, by Corollary \ref{cor: condition on height functions in terms of group action}, there exists a 
zygothety \((\lambda,\phi) = ((\lambda_1,\lambda_2),(\phi_1,\phi_2))\)
such that \((g_+,g_-)\circ(\lambda,\phi) = (f_+,f_-)\).
Suppose that neither \(F\) nor \(G\) has \(Y\) as a factor. If \(r\) is odd and \(s\) is even, then \(\lambda_1 = \lambda_2\) (by Lemma \ref{lemma: r odd, s even}). Otherwise, \(r\) is even or \(s\) is odd, and by Theorem \ref{thm: Sufficient conditions for R-semialg. Lip. equivalence}, \(F\) and \(G\) are \({\cal R}\)-semialgebraically Lipschitz equivalent. This proves part (a) of the corollary. 

Now, we prove part (b). Suppose that one of the height functions \(f_+,f_-,g_+,g_-\) has only one critical point. By Theorem \ref{thm: Sufficient conditions for R-semialg. Lip. equivalence}(a), if \(r\) is even or \(s\) is odd, then \(F\) and \(G\) are \({\cal R}\)-semialgebraically Lipschitz equivalent. Thus, assume that \(r\) is odd and \(s\) is even. By the (already proved) part (a) of this corollary, if neither \(F\) nor \(G\) has \(Y\) as a factor, then \(F\) and \(G\) are \({\cal R}\)-semialgebraically Lipschitz equivalent. So, we add to our assumptions that \(Y\) is a factor of both \(F\) and \(G\). (Since we are assuming that \(r\) is odd and \(s\) is even, Lemma \ref{lemma: r odd, s even}(i) guarantees that either both \(F\) and \(G\) have \(Y\) as a factor, or neither does.) Then, by Lemma \ref{lemma: r odd, s even}(iii), there exists a zygothety \((\tilde\lambda,\tilde\phi)=
((\tilde\lambda_1,\tilde\lambda_2),(\tilde\phi_1,\tilde\phi_2))\),
with \(\tilde\lambda_1=\tilde\lambda_2\), such that
\((g_+,g_-)\circ(\tilde\lambda,\tilde\phi)=(f_+,f_-)\) and, by Theorem \ref{thm: Sufficient conditions for R-semialg. Lip. equivalence}(b), \(F\) and \(G\) are \({\cal R}\)-semialgebraically Lipschitz equivalent.
\end{proof}

\subsection{Henry and Parusi\'{n}ski's example revisited}
\label{subsection: Henry and Parusinski's example}

In \cite{HP2}, Henry and Parusi\'{n}ski constructed an invariant of the bi-Lipschitz equivalence of analytic function germs 
\((\R^n,0)\to(\R,0)\) that varies continuously in many analytic families, thus showing that the bi-Lipschitz equivalence of analytic function germs admits continuous moduli. As an example, they show that the one-parameter family of germs \(f_t(x,y)\from (\R^2,0)\to(\R,0)\), given by
\begin{equation}
\label{eq: Henry-Parusinski's example}
f_t(x,y) = x^3 - 3txy^4 + y^6,\quad t\in\R
\end{equation}
satisfies the following properties:
\begin{enumerate}[label=\roman*.]
\item Given \(t,t^\prime > 0\), if \(t\neq t^\prime\) then there exists no germ of bi-Lipschitz homeomorphism \(h\from (\R^2,0)\to(\R^2,0)\) such that
\(f_t\circ h = f_{t^\prime}\).

\item Given \(t,t^\prime < 0\), there exists a germ of bi-Lipschitz homeomorphism \(h\from (\R^2,0)\to(\R^2,0)\) such that
\(f_t\circ h = f_{t^\prime}\).
\end{enumerate}
Note that property (i) shows, in particular, that the bi-Lipschitz classification of real analytic function germs admits continuous moduli.

As an application of the results obtained in this paper, we analyze this example in the context of \({\cal R}\)-semialgebraic Lipschitz equivalence. But before we do that, we need to make a small adjustment. For each \(t\in\R\), we have \(f_t(\lambda^2 x,\lambda y) = \lambda^6f_t(x,y)\). Hence, \(f_t(x,y)\) is quasihomogeneous with weights \(2\) and \(1\). However, we established in Section \ref{section: intro} that we would work with quasihomogeneous polynomials \(F(X,Y)\) in which \(Y\) has a greater weight than \(X\). So, in order to apply our results on \({\cal R}\)-semialgebraic Lipschitz equivalence we work instead with the family
\begin{equation}
\label{eq: Henry-Parusinski's example modified}
F_\lambda (X,Y) = X^6 - 3\lambda X^4Y + Y^3,\quad \lambda\in\R,
\end{equation}
obtained by swapping \(x\) and \(y\) in the quasihomogeneous polynomials from the original family to make the weight of \(y\) greater than the weight of \(x\). Thus, we are working now with a family of \(\beta\)-quasihomogeneous polynomials of degree \(6\), with \(\beta = 2\). We show that this family satisfies the following properties:\footnote{Since the map \(\R^2 \to \R^2\) given by \((X,Y)\mapsto (Y,X)\) is a semialgebraic bi-Lipschitz homeomorphism,  these properties are also satisfied by the family (\ref{eq: Henry-Parusinski's example}).}
\begin{enumerate}[label=\roman*'.]
\item Given \(\lambda,\mu > 0\), if \(\lambda\neq\mu\) then there exists no germ of semialgebraic bi-Lipschitz homeomorphism \(\Phi\from (\R^2,0)\to(\R^2,0)\) such that \(F_\mu\circ \Phi = F_\lambda\).

\item Given \(\lambda,\mu < 0\), there exists a germ of semialgebraic bi-Lipschitz homeomorphism \(\Phi\from (\R^2,0)\to(\R^2,0)\) such that
\(F_\mu\circ \Phi = F_\lambda\).
\end{enumerate}

Note that the height functions of \(F_\lambda\) are given by
\((f_\lambda)_+(t) = 1 - 3\lambda t + t^3\) and \((f_\lambda)_-(t) = 1 - 3\lambda t + t^3\). So, we drop the subscript sign and write simply 
\(f_\lambda(t) = t^3 - 3\lambda t + 1\). 
Now, to prove (i'), fix any two real numbers \(\lambda,\mu > 0\); we show that if \(F_\lambda\) and \(F_\mu\) are \({\cal R}\)-semialgebraically Lipschitz equivalent then \(\lambda = \mu\). We proceed in two steps. First, we show that if \(F_\lambda\) and \(F_\mu\) are 
\({\cal R}\)-semialgebraically Lipschitz equivalent then 
\(f_\lambda\cong f_\mu\). Second, we show that if \(f_\lambda\cong f_\mu\) then \(\lambda = \mu\).

For the first step, note that  \(f_\lambda\) has at least one real zero \(t_0\) and that \(X\) is not a factor of \(F_\mu\).
Hence, by Theorem \ref{thm: equivalent polynomials, Lipschitz equivalent height functions}, the \({\cal R}\)-semialgebraic equivalence of \(F_\lambda\) and \(F_\mu\) implies that \(f_\lambda\cong f_\mu\). 

For the second step, note that \(f_\lambda\) has exactly two distinct critical points, \(t_1 = -\lambda^{1/2}\) and \(t_2 = \lambda^{1/2}\), and  that its multiplicity symbol is \(((1+2\lambda^{3/2},1-2\lambda^{3/2}),(2,2))\). Also, \(f_\mu\) has exactly two distinct critical points \(s_1 = -\mu^{1/2}\) and \(s_2 = \mu^{1/2}\) and its multiplicity symbol is \(((1+2\mu^{3/2},1-2\mu^{3/2}),(2,2))\).
Now, suppose that \(f_\lambda\cong f_\mu\). By Theorem \ref{thm: at least 2 critical points}, the multiplicity symbols of \(f_\lambda\) and \(f_\mu\) are similar. More precisely, since \(1+2\lambda^{3/2} > 1-2\lambda^{3/2}\) and
\(1+2\mu^{3/2} > 1-2\mu^{3/2}\), the multiplicity symbols of \(f_\lambda\) and \(f_\mu\) are directly similar, and hence
\begin{equation*}
\left|
\begin{array}{ll}
1+2\lambda^{3/2} & 1-2\lambda^{3/2}\\
1+2\mu^{3/2} & 1-2\mu^{3/2}
\end{array}
\right|
= 0\,.
\end{equation*}
A straightforward verification shows that this equality holds if and only if \(\lambda = \mu\), so we conclude that 
if \(f_\lambda\cong f_\mu\) then \(\lambda = \mu\).

To prove (ii'), fix any two real numbers \(\lambda,\mu < 0\).
In this case, both \(f_\lambda\) and \(f_\mu\) have no critical points.
Hence, by Corollary \ref{cor: Sufficient conditions for R-semialg. Lip. equivalence}, \(F_\lambda\) and \(F_\mu\) are 
\({\cal R}\)-semialgebraically Lipschitz equivalent.

Finally, we note that property (i') shows, in particular, that the 
\({\cal R}\)-semialgebraic Lipschitz equivalence of 
real \(\beta\)-quasihomogeneous polynomials in two variables admits continuous moduli.\\


\noindent
{\bf Acknowledgements.} I would like to thank Alexandre Fernandes, Vincent Grandjean and Maria Michalska for their guidance and support in completing this research.\\ 

\noindent
{\bf Funding.} This work was financially supported by Funcap (Fundação Cearense de Apoio ao Desenvolvimento Científico e Tecnológico) and UESC (Universidade Estadual de Santa Cruz).

\bibliographystyle{amsplain}
\bibliography{semialg_lip_equiv_of_polyn_funct}

\begin{thebibliography}{1}

\bibitem{BFP}
Lev BIRBRAIR, Alexandre César~Gurgel FERNANDES, and Daniel PANAZZOLO.
\newblock {Lipschitz classification of functions on a H{\"o}lder triangle}.
\newblock {\em {\bf St. Petersburg Mathematical Journal}}, 20(5):681--686,
  2009.
\newblock Available at:
  \url{https://www.ams.org/journals/spmj/2009-20-05/S1061-0022-09-01067-X/home.html}.
  Access~in:~Mar.~07~2025.

\bibitem{CR}
Leonardo~Meireles CÂMARA and Maria Aparecida~Soares RUAS.
\newblock {On the Moduli Space of Quasi-Homogeneous Functions}.
\newblock {\em {\bf Bull. Braz. Math. Soc., New Series}}, 53:895--908, 2022.
\newblock Available at: \url{https://doi.org/10.1007/s00574-022-00287-8}.
  Access~in:~Mar.~07~2025.

\bibitem{Drago}
Silvestru~S. DRAGOMIR.
\newblock {A survey on Ostrowski type inequalities related to Pompeiu's mean
  value theorem}.
\newblock {\em {\bf Khayyam J. Math.}}, 1(1):1--35, 2015.
\newblock Available at: \url{http://www.kjm-math.org/article\_12284.html}.
  Access~in:~Mar.~07~2025.

\bibitem{HP2}
Jean-Pierre HENRY and Adam PA\-RU\-SI\'{N}S\-KI.
\newblock {Invariants of bi-Lipschitz equivalence of real analytic functions}.
\newblock {\em {\bf Banach Center Publications}}, 65:67--75, 2004.
\newblock Available at:
  \url{https://www.impan.pl/en/publishing-house/banach-center-publications/all/65/0/86125/invariants-of-bi-lipschitz-equivalence-of-real-analytic-functions}.
  Access~in:~Mar.~07~2025.

\bibitem{HP1}
Jean-Pierre HENRY and Adam PARUSI\'{N}SKI.
\newblock {Existence of moduli for bi-Lipschitz equivalence of analytic
  functions}.
\newblock {\em {\bf Compositio Mathematica}}, 136(2):217--235, 2003.
\newblock Available at:
  \url{https://www.cambridge.org/core/journals/compositio-mathematica/article/existence-of-moduli-for-bilipschitz-equivalence-of-analytic-functions/4E7B7D3CE5764ACA0A28DEF2211378C1}.
  Access~in:~Mar.~07~2025.

\bibitem{KP}
Satoshi KOIKE and Adam PARUSI\'{N}SKI.
\newblock {Equivalence relations for two variable real analytic function
  germs}.
\newblock {\em {\bf J. Math. Soc. Japan}}, 65(1):237--276, 2013.
\newblock Available at: \url{https://projecteuclid.org/euclid.jmsj/1359036454}.
  Access~in:~Mar.~07~2025.

\bibitem{M}
Tadeusz MOSTOWSKI.
\newblock {\em {Lipschitz equisingularity}}.
\newblock Instytut Matematyczny Polskiej Akademi Nauk, 1985.
\newblock Available at: \url{https://eudml.org/doc/268625}.
  Access~in:~Mar.~07~2025.

\end{thebibliography}









\vspace{10pt}
\noindent
{\bf Sergio Alvarez A. Correia}\\
Departamento de Ciências Exatas,\\
Universidade Estadual de Santa Cruz,\\
Ilhéus-Bahia, Brazil;\\
salvarez@uesc.br 

\end{document}